\documentclass[12pt]{amsart}
\usepackage{fullpage,amsmath,amssymb,amsthm,enumerate,hyperref}
\usepackage{graphicx}
\hypersetup{colorlinks=true,  
    linkcolor=blue,          
    citecolor=blue,        
    filecolor=black,      
    urlcolor=black           
    }

\newtheorem{thm}{Theorem}
\newtheorem{lem}[thm]{Lemma}
\newtheorem{cor}[thm]{Corollary}
\newtheorem{claim}[thm]{Claim}
\newtheorem{step}{Step}
\newtheorem*{ansatz}{Ansatz}

\newcommand{\Z}{\mathbb Z}
\newcommand{\N}{\mathbb N}
\newcommand{\R}{\mathbb R}

\newcommand{\bbp}{\mathbb{P}}

\DeclareMathOperator{\cov}{Cov}

\newcommand{\tld}[1]{\tilde{#1}}

\newcommand{\ep}{\epsilon}
\newcommand{\del}{\delta}
\newcommand{\om}{\omega}

\newcommand{\lam}{\lambda}

\newcommand{\sig}{\sigma}

\newcommand{\oom}{\bar{\omega}}
\newcommand{\ssig}{\bar{\sigma}}
\newcommand{\C}{\mathbb{C}}

\newcommand{\PP}{\mathbb{P}}

\newcommand{\cocycle}[2]{A^{(#1)}_{#2}}
\newcommand{\cocyclestar}[2]{A^{(#1)\,*}_{#2}}
\newcommand{\cocycleps}[2]{{\pertA}^{(#1)}_{#2}}
\newcommand{\pertA}{A^\epsilon}

\newcommand{\topspacemat}[2]{F_{#1}(#2)}
\newcommand{\topspaceDSunpert}[2]{F_{#1}(#2)}
\newcommand{\topspaceDSpert}[2]{\tilde F_{#1}(#2)}
\newcommand{\topspacestarDSunpert}[2]{F^*_{#1}(#2)}

\newcommand{\bottomspacemat}[2]{E_{#1}(#2)}
\newcommand{\bottomspaceDSunpert}[2]{E_{#1}(#2)}
\newcommand{\bottomspaceDSpert}[2]{\tilde E_{#1}(#2)}
\newcommand{\bottomspacestarDSunpert}[2]{E^*_{#1}(#2)}

\begin{document}

\title[Stochastic stability of Oseledets decompositions]{Stochastic stability of Lyapunov exponents and Oseledets splittings for semi-invertible matrix cocycles} \author{Gary
  Froyland, Cecilia Gonz\'alez-Tokman and Anthony Quas}
\address[Froyland and Gonz\'alez-Tokman]{School of Mathematics and Statistics, University of New South Wales,
  Sydney, NSW, 2052, Australia} \address[Quas]{Department of Mathematics and
  Statistics, University of Victoria, Victoria, BC, CANADA, V8W 3R4}

\begin{abstract}
 We establish (i) stability of Lyapunov exponents and (ii) convergence in probability of Oseledets spaces for semi-invertible matrix cocycles, subjected to small random perturbations. 
The first part extends results of Ledrappier and Young \cite{LedrappierYoung} to the semi-invertible setting.
The second part relies on the study of evolution of subspaces 
 in the Grassmannian, where the analysis developed is likely to be of wider interest.
\end{abstract}

\maketitle
\section{Introduction}

The landmark Oseledets Multiplicative Ergodic Theorem (MET) plays a
central role in modern dynamical systems, providing a basis for the
study of non-uniformly hyperbolic dynamical systems. Oseledets'
theorem has been extended in many ways beyond the original context of
products of finite-dimensional matrices, for instance to certain
classes of operators on Banach spaces and more abstractly to
non-expanding maps of non-positively curved spaces.

The original Oseledets theorem \cite{Oseledets} was formulated in both an invertible
version (both the base dynamics and the matrices are assumed to be
invertible) and a non-invertible version (neither the base dynamics
nor the matrices are assumed to be invertible). The conclusion in the
non-invertible case is much weaker than in the invertible case: in the
invertible version, the theorem gives a \emph{splitting} (that is, a
direct sum decomposition) of $\R^d$ into equivariant subspaces, each
with a characteristic exponent that is used to order the splitting components from largest to smallest expansion rate; whereas in the non-invertible version,
the theorem gives an equivariant \emph{filtration} (that is, a
decreasing nested sequence of subspaces) of $\R^d$.

In various combinations, the current authors and collaborators have
been working on extensions of the MET to what we have called the
\emph{semi-invertible} setting \cite{FLQ1,FLQ2,GTQuas}. This refers to the assumption that one
has an invertible underlying base dynamical system (also known as driving or forcing), but that the
matrices or operators that are composed may fail to be
invertible. In this setting, our theorems yield an equivariant
splitting as in the invertible case of the MET, rather than the
equivariant filtration that the previous theorems would have given.

We are interested in applications where the operators are
Perron-Frobenius operators of dynamical systems acting on suitable
Banach spaces. Here, the `suitable' Banach spaces are spaces that are
mapped into themselves by the Perron-Frobenius operator, and on which
the Perron-Frobenius operator is quasi-compact.  These Banach spaces
have been widely studied in the case of a single dynamical system.

An Ansatz that first appeared in a paper of Dellnitz, Froyland and
Sertl \cite{DellnitzFroylandSertl} in the context of Perron-Frobenius operators of a single
dynamical system is the following:

\begin{ansatz}The peripheral spectrum (that is, spectrum of the
Perron-Frobenius operator outside the essential spectral radius)
corresponds to global features of the system (such as bottlenecks or
almost-invariant regions) whereas the essential spectrum corresponds
to local features of the system, such as rates of expansion.
\end{ansatz}

In a series of papers, they take this idea further by showing that
level sets of eigenfunctions with eigenvalues peripheral to the
essential spectral radius can be used to locate almost-invariant sets in the dynamical system
\cite{DellnitzJunge,FroylandDellnitz,Froyland2005,GTHuntWright}.
Figure~\ref{fig1} gives a schematic illustration of such a system:
The left and right halves are almost-invariant under the dynamics, but the bottleneck joining them allows small but non-negligible interaction between them.

\begin{figure}[hbt]
\begin{center}
 \begin{tabular}{cc}
       \includegraphics[height=3.25cm]{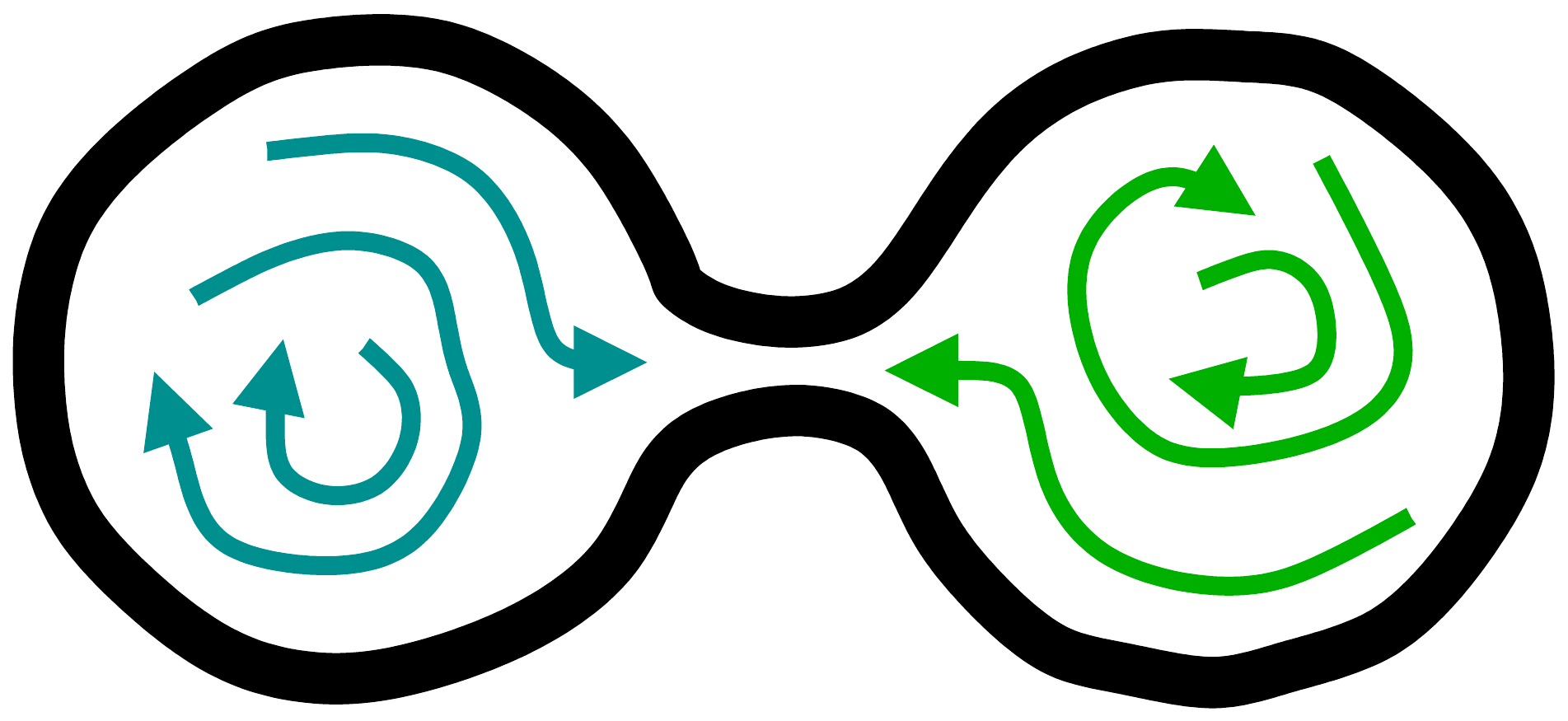}
 & \includegraphics[height=3.25cm]{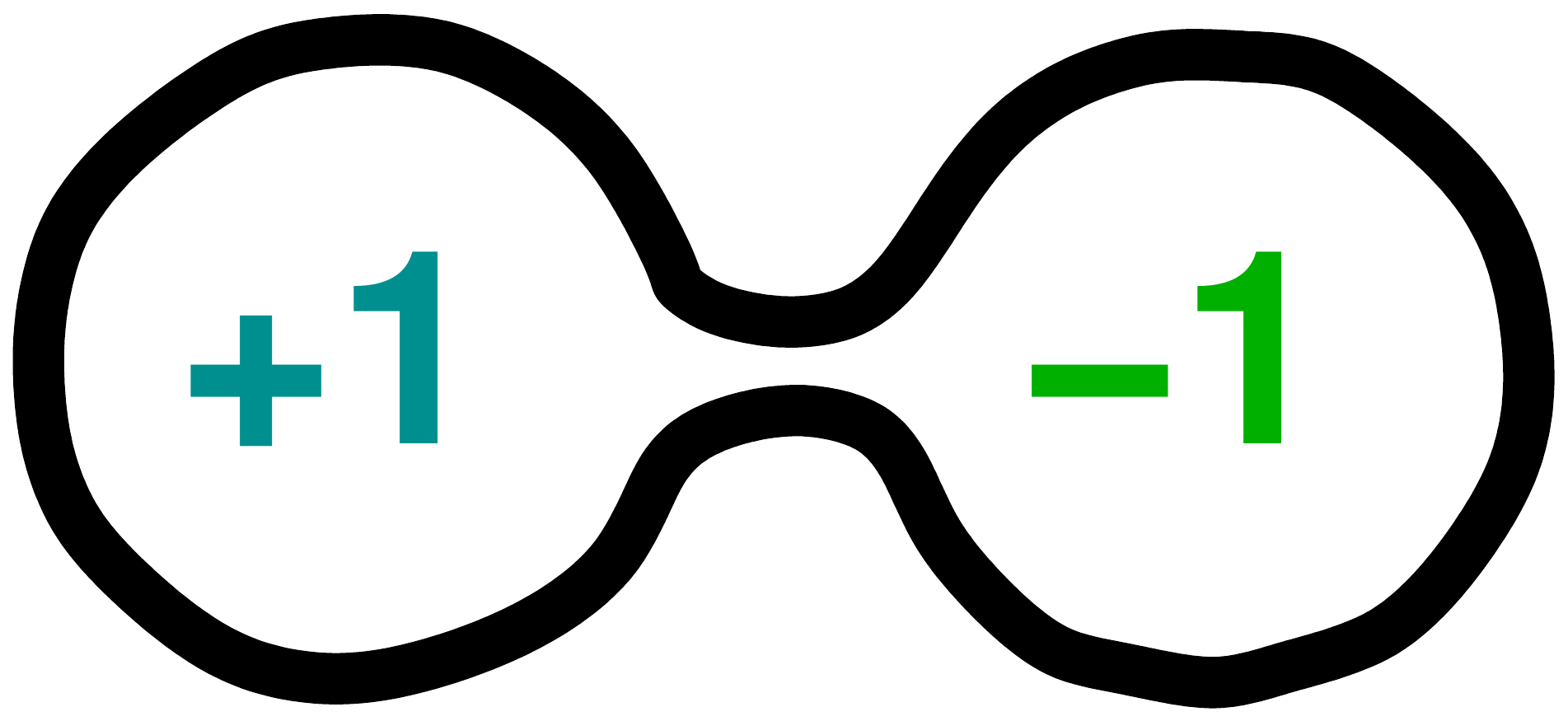} 
 \\
     {\footnotesize (a) } 
& {\footnotesize (b)} 
  \end{tabular}
  \end{center}
  \caption{(a) Schematic representation of a dynamical system with almost-invariant regions. 
  (b) Approximate values of eigenfunction corresponding to the bottleneck.}\label{fig1}
\end{figure}

While this Ansatz was initially made in the context of a single
dynamical system, it seems to apply equally in the case of random
dynamical systems \cite{FroylandLloydSantitissadeekorn,FroylandSantiMonahan,Froyland-AnalyticFramework}, 
and this is the central motivation for our research
in this area. It is well known that Perron-Frobenius operators of non-invertible maps are
essentially never invertible, but it is often reasonable to assume
that the base dynamics are invertible. Indeed, even if the driving system  
is non-invertible,
one can make use of canonical mathematical techniques to extend it to an invertible one.
  Hence, we naturally find
ourselves in the semi-invertible category.  The principal object that
we are interested in understanding is the second Oseledets subspace
(or more generally the first few Oseledets subspaces).

The significance of our extensions to the MET is that the second
subspace that we obtain is low-dimensional (typically one-dimensional)
instead of $(d-1)$-dimensional, which is what would come from the
standard non-invertible MET. In numerical applications, where $d$ may
be $10^5$ or greater, one cannot expect to say anything reasonable
about level sets of functions belonging to a high-dimensional
subspace, whereas using the semi-invertible version of the theorem, we
are once again in a position to make sense of the level sets.

In practice, of course, one cannot numerically study the action of
Perron-Frobenius operators on infinite-dimensional Banach spaces. Nor
can one find a finite-dimensional subspace preserved by the
operators. A remarkably fruitful approach is the so-called \emph{Ulam
  method}. Here, the state space is cut into small reasonably regular
pieces and a single dynamical system is treated as a Markov chain, by
applying the dynamical system and then randomizing over the cell in
which the point lands. This also makes sense for random dynamical
systems. 

In \cite{FGTQ}, we showed that applying the Ulam method to certain random
dynamical systems, the top Oseledets space of the truncated system
converges in probability to the true top Oseledets space of the random
dynamical system as the size of the partition is shrunk to 0. The top
Oseledets space is known to correspond to the random absolutely
continuous invariant measure of the system. It is natural to ask
whether the subsequent Oseledets spaces for the truncated systems
converge to the corresponding Oseledets spaces for the full system.
We are not yet able to answer this, although the current paper
represents a substantial step in this direction.

In \cite{FGTQ}, we viewed the Ulam projections of the Perron-Frobenius
operator as perturbations of the original operator, and showed that
the top Oseledets space was robust to the kind of projections that
were being considered. In this paper, we prove convergence the
of subsequent Oseledets spaces under certain perturbations, but do this
in the context of matrices instead of infinite-dimensional operators. 

In general, Lyapunov exponents and Oseledets subspaces are known to be
highly sensitive to perturbations. A mechanism responsible for this
is attributed to Ma\~n\'e; see also \cite{ArnoldCong-simple}. 
Ledrappier and Young considered the case of perturbations of
random products of uniformly invertible matrices \cite{LedrappierYoung}. 
This followed related work of Young in the two-dimensional setting \cite{Young86}.
In view of the
sensitivity results, it was necessary to restrict the class of
perturbations that they considered, and they dealt with the situation
where the distribution of the perturbation of the matrix at time 0 was
absolutely continuous (with control on the density) conditioned on all
previous perturbations. The simplest instance of this is the case where
the matrices to be multiplied are subjected to additive
i.i.d. absolutely continuous noise. In this situation, they showed
that the perturbed exponents converge almost surely to the true
exponents as the noise is shrunk to 0. While they did not directly
address the Oseledets subspaces, work of Ochs shows that convergence
of the Lyapunov exponents implies convergence in probability of the
Oseledets subspaces \cite{Ochs}. 

In this paper, we deal with the case of uniform i.i.d. additive noise
in the matrices, but make no assumption on invertibility. 
The conclusions that we obtain are the same as may be
obtained in the invertible case. Our argument first demonstrates
stability of the Lyapunov exponents, and then shows stability of the
Oseledets subspaces. The first part is closely based
on Ledrappier and Young's approach, although we need to do some
non-trivial extra work to deal with the lack of uniform invertibility
(in Ledrappier-Young's argument, in one step, there is an upper bound to the amount
of damage that can be done to the exponents, whereas in the
non-invertible case there is no such bound). The second part of the
argument is completely new. The methods of Ochs cannot be made to work
here because they are based on finding the space with the smallest
exponent and then using tensor products to move up the ladder. In the
case where the smallest exponent is $-\infty$, when one takes tensor
products, all products with this subspace have exponent $-\infty$ so
there is no distinguished second-smallest subspace. To get around this
problem, we use the Grassmannian in place of the exterior algebra. We
study evolution of subspaces in the Grassmannian, and show that this
is controlled by fractional linear transformations. An important role
is played by a higher-dimensional analogue of the cross ratio.

We are hopeful that the techniques that we introduce to control
evolution of these Oseledets subspaces under the matrices may be
applied much more widely.

\subsection{Statements of the main results}

If $\sigma\colon (\Omega,\PP)\to (\Omega,\PP)$ is a
measure-preserving transformation of a probability space and $A\colon
\Omega\to M_{d\times d}(\R)$ is a measurable matrix-valued function, 
we let $\cocycle n\omega$ denote the product
$A(\sigma^{n-1}\omega)A(\sigma^{n-2}\omega)\cdots A(\omega)$.
We call the tuple $(\Omega,\PP,\sigma,A)$ a \emph{matrix cocycle}.

Let $U=\{A\in M_{d\times d}(\R)\colon \|A\|\le 1\}$ (here and throughout the paper,
the norm of a matrix is its operator norm). Let $\bar\Omega=\Omega\times U^\Z$.
We write $\bar\omega=(\omega,\Delta)$ for a typical element of $\bar\Omega$.
If $\Omega$ is equipped with the measure $\PP$, we equip 
$\bar\Omega$ with the measure
$\bar\PP=\PP\times\lambda^\Z$, where $\lambda$ is the uniform measure
on $U$ and $\lambda^\Z$ is the product measure.
We fix an $\epsilon>0$. Then for an element $\bar\omega\in \bar\Omega$, 
the corresponding sequence of matrices
is $(\pertA_n(\bar\omega))_{n\in\Z}=(A(\sigma^n\omega)+\epsilon\Delta_n)_{n\in\Z}$.
This paper is concerned with a comparison of the properties of the matrix cocycle
$(\Omega,\PP,\sigma,A)$ with those of the matrix cocycle 
$(\bar\Omega,\bar \PP,\bar\sigma,\pertA)$ as $\epsilon\to 0$.

The main result of this paper is the following.

\begin{thm}\label{thm:main}
Let $\sigma$ be an ergodic measure-preserving transformation
of $(\Omega,\PP)$ and let $A\colon\Omega\to M_{d\times d}(\R)$ be
a measurable map such that $\int\log^+\|A(\omega)\|\,d\PP(\omega)<\infty$.

Let the Lyapunov exponents of the matrix cocycle be 
$\lambda_1>\ldots>\lambda_p\ge-\infty$ with multiplicities $d_1,d_2,\ldots,d_p$
and let the corresponding Oseledets decomposition be
$\R^d=Y_1(\omega)\oplus\ldots\oplus Y_p(\omega)$.

Let $D_0=0$, $D_i=d_1+\ldots+d_i$ and let the Lyapunov exponents (with multiplicity)
be $\infty>\mu_1\ge \mu_2\ge\ldots\ge \mu_d\ge-\infty$, so that $\mu_i=\lambda_k$
if $D_{k-1}<i\le D_k$.

\begin{enumerate}[(I)]
\item (Convergence of Lyapunov exponents)\label{part:expts}
Let the Lyapunov exponents of the perturbed matrix cocycle $(\bar\Omega,\bar P,
\bar\sigma,\pertA)$ (with multiplicity) be $\mu_1^\ep\ge \mu_2^\ep\ge\ldots\ge 
\mu_d^\ep$. Then $\mu_i^\ep\to\mu_i$ for each $i$ as $\ep\to 0$.

\item (Convergence in probability of Oseledets spaces)\label{part:spaces}
Let $\tau=\frac1{12}\min_{1\le i<j}(\lambda_i-\lambda_{i+1})$. 
Let $\epsilon_0$ be such 
that $|\mu_i^\ep-\mu_i|<\tau$ for each $i$ for all $\epsilon\le \epsilon_0$.
For $\epsilon<\epsilon_0$, let $Y^\ep_i(\omega)$ denote the sum of the Lyapunov
subspaces having exponents in the range $(\lambda_i-\tau,\lambda_i+\tau)$.
Then $Y^\ep_i(\bar \omega)$ converges in probability to $Y_i(\omega)$ as $\ep\to 0$.

\end{enumerate}
\end{thm}

\subsection{Outline of the paper}
Section \ref{sec:prelim} introduces terminology, background results and a collection of lemmas that will be used in the proof of the main result. Theorem~\ref{thm:main}\eqref{part:expts} is established in Section~\ref{sec:Exponents}, and part \eqref{part:spaces} is proven in Section~\ref{sec:Spaces}.

\section{Preliminaries}\label{sec:prelim}

For two subspaces, $U$ and $V$, of $\R^d$ of the same dimension, we define
$\angle(U,V)=d_H(U\cap B,V\cap B)$, where $d_H$ denotes Hausdorff
distance and $B$ is the unit ball. For two subspaces $U$ and $W$ of
complementary dimensions, we define $\perp(U,W)=(1/\sqrt
2)\inf_{\{u\in U\cap S,\,w\in W\cap S\}}\|u-w\|$, where $S$ denotes
the unit sphere.  Thus $\perp(U,W)$ is a measure of complementarity of
subspaces, taking values between 0 and 1, with 0 indicating that the
spaces intersect and 1 indicating that the spaces are orthogonal
complements. Note that $\perp(U,V)\geq \perp(U,W)-\angle(W,V)$.

Let $s_j(A)$ denote the $j$th singular value of the matrix $A$ and let
$\Xi^j(A)$ denote $\log s_1(A)+\ldots+\log s_j(A)$. Note that
$\Xi^j(A)=\log \|\Lambda^j A\|$, so that $\Xi^j(AB)\le \Xi^j(A)+\Xi^j(B)$.

The structure of the proof of the main theorem closely follows that of
Ledrappier and Young, in which the orbit of $\omega$ is divided 
into blocks of length
$\approx|\log\epsilon|$. These are classified as good if a number 
of conditions hold (separation
of Lyapunov spaces, closeness of averages to integrals etc.) and bad otherwise.
The crucial modifications that we make are in estimations
for the bad blocks. In the case of \cite{LedrappierYoung}, the matrices 
(and hence their perturbations) have
uniformly bounded inverses, so that for bad blocks one can 
give uniform lower bounds on
the contribution to the singular value. By contrast, here, 
there is no uniform lower bound. Upper bounds are straightforward, so all
of the work is concerned with establishing lower bounds for the exponents. 
Absent the invertibility, a similar argument would yield
(random) bounds of order $\log\epsilon$, which turn out to be too
weak to give the lower bounds that we need.

Given a matrix $A$ with the property that $s_{j+1}(A)<s_j(A)$, we
define $\bottomspacemat jA$ to be the space spanned by the $(j+1)$st to $d$th
singular vectors and $\topspacemat jA$ to be the space spanned by the images of
the 1st to $j$th singular vectors under $A$.
If one has a matrix cocycle with base space $\Omega$ and matrices $A_\omega$,
we use the very similar notation $\bottomspaceDSunpert j\omega$ and 
$\topspaceDSunpert j\omega$ to refer to the Oseledets subspaces. 
The convention will be that if the argument is a matrix, then
they refer to the span of the bottom singular vectors or the 
images of the top singular vectors,
while if the argument is a point of the base space, 
they refer to spaces appearing in the Oseledets
theorem. These spaces are obtained simply as limits of spans of singular vectors, 
as explained in Lemma \ref{lem:SVs} below, justifying the notation.

\begin{lem}[Singular Vectors of blocks of the unperturbed system]
\label{lem:SVs}
  Suppose that the unperturbed system has exponents satisfying
  $\lambda_{j} >\lambda_{j+1}$.

  Then for almost every $\omega$, $\bottomspacemat j{\cocycle n \omega}\to 
  \bottomspaceDSunpert j\omega$ and 
  $\topspacemat j{\cocycle n{\sigma^{-n}\omega}}\to
  \topspaceDSunpert j\omega$ as $n\to\infty$.
\end{lem}

\begin{proof}
The statement that $\bottomspacemat j{\cocycle n{\omega}}\to 
\bottomspacemat j\omega$
follows from the proof of Oseledets' theorem given in \cite{ArnoldRDSBook}.
The singular value decomposition ensures that 
$\topspacemat j{\cocycle n{\sigma^{-n}\omega}}=
\Big( \bottomspacemat j{\cocyclestar n{\sigma^{-n}\omega}} \Big)^\perp$, 
where $A^*$ denotes the adjoint of $A$.

On the other hand, a similar statement is true for the spaces 
$\bottomspacestarDSunpert j\omega$ and $\topspaceDSunpert j\omega$. 
More precisely, we claim that if we let $\bottomspacestarDSunpert j\omega$ 
be the Oseledets spaces 
for the dual cocycle with base $\sigma^{-1}$ and generator 
$G(\omega)={A(\sigma^{-1}\omega)}^*$, then
$\topspaceDSunpert j\omega=(\bottomspacestarDSunpert j\omega)^\perp$.
Applying Oseledets' theorem to the dual cocycle, we obtain
$\bottomspacemat j{\cocyclestar n{\sigma^{-n}\omega}}\to 
\bottomspacestarDSunpert j\omega$.

To prove the claim, suppose for a contradiction that there exist
$f(\omega) \in \topspaceDSunpert j\omega$ 
and $e^*(\omega) \in \bottomspacestarDSunpert j\omega$ of unit length such that 
$\langle f(\omega), e^*(\omega) \rangle\neq 0$.
By invertibility of $\cocycle n{\sigma^{-n}\omega}$ as a map from  
$\topspaceDSunpert j{\sigma^{-n}\omega}$ to $\topspaceDSunpert j \omega$
there exist for all $n$, unit vectors
$f^{(-n)}(\omega) \in \topspaceDSunpert j{\sigma^{-n}\omega}$ such that 
$\cocycle n{\sigma^{-n}\omega} f^{(-n)}(\omega)$ is a multiple of $f(\omega)$.
Then,
\begin{equation}\label{eq:duality}
 \langle \cocycle n{\sigma^{-n}\omega} f^{(-n)}(\omega), e^*(\omega) \rangle = 
 \langle f^{(-n)}(\omega),  \cocyclestar n{\sigma^{-n}\omega} e^*(\omega)
 \rangle. 
\end{equation}
The right hand side grows at a rate slower than $\lambda_j$.
The left hand side grows at a rate at least $\lambda_j$, by  \cite[Eq.~(18)]{FLQ1}.
This yields a contradiction, so 
 $\topspaceDSunpert j\omega\subset 
 \big(\bottomspacestarDSunpert j\omega\big)^\perp$.
Since the dimensions agree, they coincide.

Now the  statement that $\topspacemat j{\cocycle n{\sigma^{-n}\omega}}\to
\topspaceDSunpert j\omega$ follows directly from the fact that 
$\bottomspacemat j{\cocyclestar n{\sigma^{-n}\omega}}\to  
\bottomspacestarDSunpert j\omega$,
and continuity of $V\mapsto V^\perp$.

\end{proof}

\begin{lem}\label{lem:LY}[Lemmas 3.3, 3.6 \& 4.3, \cite{LedrappierYoung}] 
  For any $\delta>0$, there exists a $K$ such that if (i) the ratio of the
  $j$th to $(j+1)$st singular values of a matrix $A$ exceeds $K$; (ii)
  the $j$th singular value exceeds $K$; and (iii) $\|B-A\|\le 1$, then the
  following hold:
  \begin{enumerate}[(a)]
  \item
  $\angle (\topspacemat jA,\topspacemat jB)$ and 
  $\angle(\bottomspacemat jA,\bottomspacemat jB)$ are less than
    $\delta/3$;\label{it:spacecont}
  \item
    $\frac13\le s_i(A)/s_i(B)\le 3$ for each $i\le j$;
  \item If $V$ is any subspace of dimension $j$ such that
  $\perp(V, \bottomspacemat jA)>\delta/6$, then 
  $\angle(BV,\topspacemat jA)<\delta/3$;\label{it:contr}
  \item If $V$ is a subspace of dimension $j$ and
    $\perp(V,\bottomspacemat jA)>\delta$, then $|\det(A|V)|\ge
   (D\delta)^j\exp\Xi_j(A)$, where $D$ is an absolute constant.\label{it:det}
  \end{enumerate}
\end{lem}

\begin{lem}\label{lem:Clogeps}
Let $\sigma$ be an ergodic measure-preserving transformation of $(\Omega,\PP)$
and let $A\colon\Omega\to M_{d\times d}(\R)$ be a measurable map such that
$\log^+\|A(\om)\|$ is integrable. There exists $C$ such that for all
$\eta_0>0$, there exists $\ep_0$ such that for all $\epsilon<\epsilon_0$,
there exists $G\subseteq\Omega$ of measure at least $1-\eta_0$
such that for all $\omega\in G$, and all 
$(\Delta_n)\in U^\Z$
\begin{equation*}
\|\cocycleps N{(\om,(\Delta_n))}-\cocycle N\om\|\le1,
\end{equation*}
where $N=\lfloor C|\log\ep|\rfloor$.
\end{lem}

\begin{proof}
Let $g(\om)=\log^+(\|A_\om\|+1)$ and let $C>0$ satisfy $\int g(\om)
\,d\PP(\om)<1/C$.
Notice that provided $\epsilon<1$ (and using the fact that the perturbations
have norm bounded by $\epsilon$), $\log^+ \|A_{\bar\om }^\ep\|\leq g(\om)$, and 
\begin{align*}
\|\cocycleps N\oom-\cocycle N\om\|&\le\sum_{i=0}^{N-1}
\|\cocycleps {N-i-1}{\bar\sigma^{i}\oom}(\pertA_{\bar\sigma^i\oom}-A_{\sigma^i\om})
\cocycle i\om\|\\
&\le N\ep\exp(g(\om)+\ldots+g(\sigma^{N-1}\om)).
\end{align*}
There exists $n_0$ such that for $N\ge n_0$, 
$N\ep\exp(g(\om)+\ldots+g(\sigma^{N-1}\om))
\le \ep \exp(N/C)$ on a set of measure at least $1-\eta_0$.
In particular, provided $\lfloor C|\log\ep_0|\rfloor>n_0$, 
taking $N=\lfloor C|\log\ep|\rfloor$, the conclusion follows.

\end{proof}

\begin{lem}\label{lem:goodblocks}
Let $\sigma$ be an ergodic measure-preserving transformation of $(\Omega,\PP)$ and
let $A\colon\Omega\to M_{d\times d}(\R)$ be a 
measurable map such that $\int\log^+\|A(\om)\|
\,\PP(\om)<\infty$. Let the Lyapunov exponents 
(with multiplicity) be $\infty>\mu_1\ge
\ldots\ge \mu_d\ge-\infty$. Suppose further that 
$\mu_j>\max(0,\mu_{j+1})$.

Let $\eta_0>0$ and $\delta_1>0$ be given. Then there exist $n_0>0$, $\kappa>0$ and
$\delta\le\min(\delta_1,\kappa)$
such that: for all $n\ge n_0$, there exists a set $G\subseteq \Omega$ with 
$\mathbb P(G)>1-\eta_0$ such that for $\omega\in G$, we have

\begin{enumerate}[(a)]
\item
$\perp(\bottomspaceDSunpert j\omega,\topspaceDSunpert j\omega)>
10\kappa$;\label{it:sep}
\item
$\angle(\topspacemat j{\cocycle n{\omega}},
\topspaceDSunpert j{\sigma^n\omega})<\delta$;\label{it:Fcont}
\item 
$\angle(\bottomspacemat j{\cocycle n\omega},\bottomspaceDSunpert j\omega
)<\delta$;\label{it:Econt}
\item
  $s_j(\cocycle n\omega)/s_{j+1}(\cocycle n\omega)>K(\delta)$ and
  $s_j(\cocycle n\omega)>K(\delta)$,
  where $K(\delta)$
  is as given in Lemma \ref{lem:LY}\label{it:svsep}. 
\end{enumerate}
\end{lem}

\begin{proof}
From the proof of Oseledets' theorem, we know $\perp(
\bottomspaceDSunpert j\omega,\topspaceDSunpert j\omega)$ is a positive
measurable function. Hence there exists $\kappa>0$ such that
\eqref{it:sep} occurs on a set of measure at least $1-\eta_0/4$.
Let $\delta=\min(\delta_1,\kappa)$.

From the proof of Oseledets' theorem, there exists an $n_1>0$ such that 
for all $n\ge n_1$, 
$\angle(\topspacemat j{\cocycle n{\sigma^{-n}\omega}},
\topspaceDSunpert j{\omega})<\delta$
and
$\angle(\bottomspacemat j{\cocycle n\omega},\bottomspaceDSunpert j\omega
)<\delta$ hold on sets of measure at least $1-\eta_0/4$. 
Hence there is a set of measure at least $1-\eta_0/4$ where \eqref{it:Econt}
holds. Similarly, using shift-invariance, there is a set of measure at least
$1-\eta_0/4$ where \eqref{it:Fcont} holds.

Since $\frac1n\log s_j(\cocycle n\om)\to\mu_j$ and
$\frac1n\log s_{j+1}(\cocycle n\om)\to\mu_{j+1}$, \eqref{it:svsep} 
holds on a set of measure at least $1-\eta_0/4$ for all $n\ge n_2$ for some $n_2>0$.
Now let $n\ge n_0=\max(n_1,n_2)$. Intersecting the above sets gives a set $G$
satisfying the conclusions of the lemma.
\end{proof}

\section{Convergence of Lyapunov exponents}\label{sec:Exponents}

\begin{proof}[Proof of Theorem \ref{thm:main}(\ref{part:expts})]

Most of the work in this part is concerned with showing the inequality
\begin{equation}
\liminf_{\ep\to 0}(\mu_1^\ep+\ldots+\mu_{D_i}^\ep)\ge \mu_1+\ldots+\mu_{D_i}
\text{,  for any $1\le i\le p$}.\label{eq:liminf}
\end{equation}
We also prove
\begin{equation}
\limsup_{\ep\to 0}(\mu_1^\ep+\ldots+\mu_j^\ep)\le \mu_1+\ldots+\mu_j \text{ for any 
$1\le j\le d$}\label{eq:limsup}
\end{equation}
which is fairly straightforward using sub-additivity. These facts,
combined with the fact that the $\mu_i^\ep$
and $\mu_i$ are decreasing in $i$ are sufficient to establish the claim
that $\mu_i^\ep\to\mu_i$ for each $i$. 

To see this, suppose that \eqref{eq:liminf} and \eqref{eq:limsup} hold. 
Let $h_j=\mu_1+\ldots+\mu_j$ and let $H_j(\epsilon)=
\mu_1^\ep+\ldots+\mu_j^\ep$. By \eqref{eq:liminf} and \eqref{eq:limsup}, 
we have $\lim_{\ep\to0}H_{D_i}(\ep)=h_{D_i}$. If $\lambda_{i+1}=-\infty$, we see
$\lim_{\ep\to 0}\mu_j^\ep=-\infty$ for all $j>D_i$ from \eqref{eq:limsup}. 
Hence we may assume that $\lambda_{i+1}>-\infty$. Since the exponents are arranged
in decreasing order, $(H_j(\ep))_{j=1}^d$ is a `concave' sequence for each
$\ep$ (that is $H_{j+1}(\ep)-H_j(\ep)\le H_j(\ep)-H_{j-1}(\ep)$
 for each $j$ in range),
as is $(h_j)_{j=1}^d$. However, $h_j$ is an arithmetic
progression for $j$ in the range
$D_i$ to $D_{i+1}-1$. Since a concave function is bounded below by its secant,
we deduce $\liminf_{\ep\to 0} H_j(\ep)\ge h_j$ for $D_i\le j<D_{i+1}$. Hence we
see $H_j(\ep)\to h_j$ as $\ep\to 0$ for each $j$, from which the statement 
follows.

To show \eqref{eq:limsup}, let $\chi>0$ and let $1\le j\le d$. By the sub-additive
ergodic theorem, there exists an $N>0$ such that $\int \Xi_j(A^{(N)}_\om)\,d\PP
<h_j+\chi/2$. Now for sufficiently small $\epsilon>0$, $\Xi_j(\cocycleps N\oom)
<\Xi_j(\cocycle N\om)+\chi/2$ for all $\oom\in\bar\Omega$. In particular, this
shows that $H_j(\ep)<h_j+\chi$ for all sufficiently small $\ep$ as required.
Notice that this part of the argument is completely general, whereas the lower
bound depends on the particular properties of the matrix perturbations.

We now focus on proving \eqref{eq:liminf}. Let $j=D_i$, noting that by the above,
we may assume that $\lambda_i>-\infty$.
By multiplying the entire family of
matrices by a positive constant, we may assume that $\mu_j=\lambda_i\ge 0$.

Let $\chi>0$ be arbitrary. 
Let $D$ be the absolute constant occurring in the statement of
Lemma \ref{lem:LY},
$C$ be as in the statement of Lemma \ref{lem:Clogeps} and $K$ be the constant occurring
in the statement of Lemma \ref{lem:trans}. Define a constant $\eta>0$ by
\begin{equation}\label{eq:etachoice}
\eta=\min\left(\frac{\chi}{4\times 1.28d^2j},\frac{\chi C}{8K}\right).
\end{equation}

Let $n_0$, $\kappa$ and $\delta$ be the quantities
given by Lemma \ref{lem:goodblocks} using $\delta_1=\frac12$ and
$\eta_0=\eta/2$. 
Since $\int\log\|A_\om\|\,d\bbp(\om)<\infty$,
it follows that $\int\Xi_j^+(A_\om)\,d\bbp(\om)<\infty$.

Let $N(\ep)=\lfloor C|\log\epsilon|\rfloor$, where $C$ is as above.
The fact that $N$ scales like 
$|\log\epsilon|$ will be of crucial importance later.
Let $\epsilon_0$ be the quantity appearing in Lemma \ref{lem:Clogeps}
with $\eta_0$ taken to be $\eta/2$.

Let $\epsilon$ be sufficiently small that
\begin{equation}\label{eq:Ncond}
\begin{split}
N(\ep)&>\frac{4j\log(3/(\delta D))}{\chi},\\
\frac{|\log\ep|}{N(\ep)}&<2/C,\\
N(\ep)&>\frac{8}\chi\int\Xi_j^+(A_\om)\,d\bbp(\om),\\
\ep&<\ep_0. 
\end{split}
\end{equation}

Let $G$ be the intersection of the good set given by Lemma \ref{lem:Clogeps}
with the good set given by Lemma \ref{lem:goodblocks} with 
$n$ taken to be $N=N(\ep)$, so that $\PP(G)>1-\eta$. If $\omega\in G$, we say
the matrix product $A_{\sigma^{N-1}\omega}\cdots A_\omega$ is a
\emph{good block}.

Now we divide everything into blocks of length $N$ and estimate the
sum of the logarithms of the first $j$ singular values of the
$\epsilon$-perturbed cocycle.

We will bound from above
the difference between the sum of the logs
of the first $j$ singular values in the unperturbed system 
and this sum in the perturbed version. We informally speak of the \emph{costs}
due to various contributions. That is, estimates of various contributions to an
upper bound for the difference
(\textsl{unperturbed})$-$(\textsl{perturbed}).
These costs are estimated in the following parts.

\begin{enumerate}[i.]
\item To deal with the concatenation of good blocks,
we give an upper bound for the difference 
(\textsl{sum of individual block exponents})$
-$(\textsl{exponent of concatenated block}). 
This is estimated using Lemma \ref{lem:LY}.
Over the whole block there is a cost of at most $\log (3/D\delta)$, so
a cost per index of $O(1/|\log\ep|)$. 
\item Reduction of singular values within bad blocks. There is an
  expected cost of at worst $1.28d^2j$ per index in a bad block
  from \eqref{eq:intbadblox}.
\item Reduction of singular values at the first and last matrix of a string of
 bad blocks. Here, there is an upper bound in expected cost of approximately
  $|\log\epsilon|$ per bad
  block. Here is where it is crucial that the blocks are of length
  $O(|\log\epsilon|)$. The upper bound for the cost averages 
  out at $O(1)$ per index 
  in each bad block. The argument is saved by the fact that most blocks are good blocks.
\end{enumerate}
The sum of the costs is $O(\eta)+O(1/|\log\ep|)$ per index
($\eta$ being the frequency 
of bad blocks), 
which will allow us to derive \eqref{eq:liminf}.
Let us proceed with the details.

\begin{step}
  A lower bound for $\Xi_j$ for concatenations of good blocks.
\end{step}
Suppose $k<l$ and 
$\sigma^{kN}\omega,\sigma^{(k+1)N}\omega,\ldots,\sigma^{(l-1)N}\omega\in
G$.  Let $B_n=\cocycle N{\sigma^{nN}\omega}$ and $\tilde
B_n=\cocycleps N{\sigma^{nN}\omega}$ We then claim that
\begin{equation}\label{eq:goodblox}
\begin{split}
  \Xi_j(\tilde B_{l-1}\cdots \tilde B_k)&\ge
  \sum_{i=k}^{l-1}\Xi_j(B_i)+(l-k)j\log\delta-(l-k)j\log(3/D)\\
  	&\ge \Xi_j(B_{l-1}\cdots B_k)+(l-k)j\log(D\delta/3),
\end{split}
\end{equation}
where $D$ is the absolute constant appearing in Lemma \ref{lem:LY}.

This is proved inductively using Lemma \ref{lem:LY}. Recall that 
$\|B_n - \tilde{B}_n\|\leq 1$. We let
$\tilde V_k=V_k=\bottomspacemat j{B_k}^\perp$ and define $V_{n+1}=B_nV_n$ and 
$\tilde V_{n+1}=\tilde B_n\tilde V_n$. 

We claim that the following hold:
\begin{enumerate}[i.]
\item $\angle(V_{n},\tilde V_{n})<\delta$ for each $n$; \label{it:ind1}
\item $\perp(V_n,\bottomspacemat j{B_n})>\delta$ and 
$\perp(\tilde V_n,\bottomspacemat j{\tilde B_n})>\delta$ for each $n$.
\label{it:ind2}
\end{enumerate}

Item \eqref{it:ind1} and the first part of \eqref{it:ind2} hold immediately 
for the case $n=k$.
The second part of \eqref{it:ind2} holds because $\tilde V_k=V_k=
\bottomspacemat j{B_k}^\perp$ and 
$\angle(\bottomspacemat j{B_k},\bottomspacemat j{\tilde B_k})<\delta$
 by Lemma \ref{lem:LY}.

Given that \eqref{it:ind1} and \eqref{it:ind2} hold for $n=m$
 and that $B_m$ is a good block,
Lemma~\ref{lem:LY} implies that $\angle(V_{m+1},\topspacemat j{B_{m}})<\delta/3$, 
$\angle(\tilde V_{m+1},\topspacemat j{\tilde B_{m}})<\delta/3$
and $\angle(\topspacemat j{B_{m}},\topspacemat j{\tilde B_{m}})<\delta/3$, 
so that $\angle(\tilde V_{m+1},V_{m+1})<\delta$, yielding
\eqref{it:ind1} for $n=m+1$.

Finally, by the induction hypothesis and Lemma~\ref{lem:LY}, we have
$\angle(\topspaceDSunpert j{\sigma^{(m+1)N}\omega},
\topspacemat j{B_{m}})<\delta$, and by Lemma~\ref{lem:goodblocks}, $\perp(\topspaceDSunpert j{\sigma^{(m+1)N}\omega},
\bottomspaceDSunpert j{\sigma^{(m+1)N}\omega)}>10\delta$. 
Thus, we obtain \eqref{it:ind2} for $n=m+1$.

Hence using Lemma \ref{lem:LY}\eqref{it:det},
we see that $\text{det}(\tilde B_n|_{\tilde V_n})\ge
(D\delta)^j\det(\tilde B_n|_{\bottomspacemat j{\tilde B_n}^\perp})=
(D\delta)^je^{\Xi_j(\tilde B_n)}\ge (D\delta/3)^je^{\Xi_j(B_n)}$.

Since $\Xi_j(\tilde B_{l-1}\cdots \tilde B_k)\ge
\prod_{i=k}^{l-1}\det(\tilde B_i|\tilde V_i)$, 
multiplying the inequalities and taking logarithms gives the result.

\begin{step}
  A lower bound for $\Xi_j$ for concatenations of arbitrary blocks.
\end{step}

Let $A_1,\ldots,A_n$ be an arbitrary sequence of matrices. 
We write
\begin{equation}\label{eq:badblox}
g^\epsilon(A_1,\ldots,A_n,\Delta_1,\ldots,\Delta_n)=
\Xi_j(\pertA_n\ldots \pertA_1)-\Xi_j(A_n\ldots A_1)
\end{equation}
and prove that $(g^\epsilon)^-$ is integrable in
$\Delta_1,\ldots,\Delta_n$ and that
\begin{equation}\label{eq:intbadblox}
\int
g^\epsilon(A_1,\ldots,A_n,\Delta_1,\ldots,\Delta_n)\,
d\lambda^n(\Delta_1,\ldots,\Delta_n)
\ge
-1.28d^2nj
\end{equation}
for all $A_1,\ldots,A_n$.

\begin{lem}\label{lem:linear}
There exists $B\approx -1.28$ such that for all $z\in\C$, and all $l\ge 0$
\begin{equation*}
\int_0^1t^l\log |1-tz|\,dt\ge B.
\end{equation*}
\end{lem}

\begin{proof}
Let us show there exists a lower bound; its precise value is irrelevant for our purposes.
Since for every $z\in \C, t\in [0,1]$, we have that $|1-tz|\geq |1-t Re(z)|\geq \big|1-t |Re(z)| \big|$, it suffices to show the lemma holds for $z\in \R^+\cup \{0\}$.
For $z=0$ the integral is 0.
Let us assume $z\in \R^+$, and let $\log^- x:=\min(0, \log x)$. Then,
\begin{align*}
 \int_0^1t^l\log |1-tz|\,dt &\geq \int_0^1t^l\log^- |1-tz|\,dt \geq \int_0^1 \log^- |1-tz|\,dt\\
& = \frac{1}{z}\int_0^z \log^- |1-y|\,dy
\geq \inf_{z\in(0,2]} \frac{1}{z}\int_0^z \log^- |1-y|\,dy.
\end{align*}
The function $g(z) :=\frac{1}{z}\int_0^z \log^- |1-y|\,dy$ for $z\neq 0$ and $g(0):=0$ is continuous on $\R^+$, and hence bounded on $[0,2]$. The statement follows.
\end{proof}

\begin{lem}\label{lem:poly}
  Let $B$ be as in Lemma \ref{lem:linear} and $p$ be a polynomial. Then, for all $l\ge 0$,
  \begin{equation*}
    \int_0^1t^l(\log|p(t)|-\log|p(0)|)\,dt\ge B\cdot\deg(p).
  \end{equation*}
\end{lem}

\begin{proof}
  If $p(0)=0$, then the result is clear. Otherwise, we consider the
  polynomial $f(t)=p(t)/p(0)$ and demonstrate that $\int_0^1
t^l \log |f(t)|\,dt\ge B\cdot\text{deg}(f)$.

  To see this, notice that $f(t)$ may be expressed as
  $\prod_{i=1}^{\text{deg}(f)}(1-tz_i)$, where $\{z_i^{-1}\}_{i=1}^{\text{deg}(f)}$ is the set of roots of $f$, and hence $p$, with multiplicity. Applying Lemma \ref{lem:linear}
  then gives the result.
\end{proof}

\begin{lem}\label{lem:operator}
  Let $B$ be the constant from the statement of
  Lemma \ref{lem:linear}. 
  Let $P(t)$ be a degree $j$ matrix-valued polynomial. That is, $P(t)$ may be
  expressed as $\sum_{k=0}^jA_kt^k$ for some collection of $d\times d$
  matrices $A_k$. Then, for all $l\ge 0$,
  \begin{equation*}
    \int_0^1t^l(\log \|P(t)\|-\log\|P(0)\|)\,dt\ge B\cdot\deg(P).
  \end{equation*}
\end{lem}

\begin{proof}
  If $P(0)$ is the zero matrix, the result is trivial. Otherwise,
  there exist unit vectors $\mathbf e$ and $\mathbf f$ such that
  $P(0)\mathbf e=\|P(0)\|\mathbf f$.

  If we set $p(t)=\langle P(t)\mathbf e,\mathbf f\rangle$, then we
  have $p(0)=\|P(0)\|$ and $\|P(t)\|\ge |p(t)|$, so the result follows from Lemma
  \ref{lem:poly}.
\end{proof}

\begin{lem}\label{lem:magicunifbd}
  Let $B\approx -1.28$ be the constant from the statement of
  Lemma \ref{lem:linear}. 
  Let $L$, $M$, $A$ and $R$ be arbitrary $d\times d$ matrices. Then
  \begin{equation*}
    \int_0^1t^l\left(\log \|\Lambda^j
    (L(A+tM)R)\|-\log\|\Lambda^j (LAR)\|\right)\,dt\ge jB.
  \end{equation*}
\end{lem}

\begin{proof}
  Notice that $\Lambda^j(L(A+tM)R)$ is a polynomial family of
  operators on $\Lambda^j\mathbb R^d$. Taking the standard orthogonal
  basis of $\Lambda^j\mathbb R^d$, let $P(t)$ be the matrix of
  $\Lambda^j (L(A+tM)R)$. The result then follows by applying Lemma
  \ref{lem:operator}.
\end{proof}

We obtain \eqref{eq:intbadblox} by a telescoping argument:
\begin{align*}
  &\Xi_j(A_n\ldots A_1)-\Xi_j(A_n^\epsilon\ldots A_1^\epsilon)\\
  &=
  \sum_{k=1}^n\left(\Xi_j(A_n^\epsilon\ldots
  A_{k+1}^\epsilon A_k\ldots A_1)-\Xi_j(A_n^\epsilon\ldots A_k^\epsilon
  A_{k-1}\ldots A_1)\right)
\end{align*}
Recall that $\Xi_j(A)=\log \|\Lambda^jA\|$. We estimate the integral of
the $k$th term in the sum. Let $L=A_n^\epsilon\ldots A_{k+1}^\epsilon$
and $R=A_{k-1}\ldots A_1$. Regarding $\Delta_{k+1},\ldots,\Delta_n$ as
fixed, we need to estimate:

\begin{equation*}
  \int_{\mathbf B}\left(\log\|\Lambda^j(LA_kR)\|-\log\|\Lambda^j(L(A_k+\epsilon
  \Delta)R)\|\right)\,d\lambda(\Delta),
\end{equation*}
where $\mathbf B=\{M\colon \|M\|\le 1\}$.
We then disintegrate the measure $\lambda$ radially, so that
$d\lambda=d^2t^{d^2-1}\,dt\cdot d(\partial\lambda)(H)$
where $\Delta=tH$, $H$ takes values in $\partial\mathbf B$ and
$\partial\lambda$ is the boundary measure. For a fixed $H$, the quantity
to estimate is
\begin{equation*}
  \int_0^1
  \left(\log\|\Lambda^j(LA_kR)\|-\log\|\Lambda^j(L(A_k+\epsilon
  tH)R)\|\right)t^{d^2-1}\,dt.
\end{equation*}

Since this quantity is uniformly bounded above, by
Lemma \ref{lem:magicunifbd}, we obtain \eqref{eq:intbadblox}.

\begin{step}Gluing blocks.
\end{step}

\begin{lem}\label{lem:trans}
Let $L$, $R$ and $A$ be given matrices. 
Then
$\Xi_j(L(A+\epsilon\Delta)R)-(\Xi_j(L)+\Xi_j(R))$ is
has integrable negative part as a
function of $\Delta$ and has integral bounded below by $K\log\epsilon$, 
where $K$ is independent of $L$, $A$ and $R$.
\end{lem}

\begin{proof}
  Write $L=O_1D_1O_2$ where $D_1$ is diagonal with entries arranged in
  decreasing order and $O_1$ and $O_2$ are orthogonal. Similarly write
  $R=O_3D_2O_4$. Let $A'=O_2AO_3$ and $\Delta'=O_2\Delta O_3$. Then we
  have
  \begin{align*}
    \Xi_j(L(A+\epsilon\Delta)R)&=\Xi_j(D_1(A'+\epsilon\Delta')D_2);\\
    \Xi_j(L)&=\Xi_j(D_1);\text{ and }\\
    \Xi_j(R)&=\Xi_j(D_2).
  \end{align*}
  Using the inequality $\Xi_j(AB)\le \Xi_j(A)+\Xi_j(B)$ and setting 
  $C$ to be the diagonal matrix with 1's in the first $j$ elements of
  the diagonal and 0's elsewhere, we have 
  \begin{align*}
    \Xi_j(D_1(A'+\epsilon\Delta')D_2)&\ge
    \Xi_j(CD_1(A'+\epsilon\Delta')D_2C)\\
    &=\Xi_j(D_1C(A'+\epsilon\Delta')CD_2)\\
    &=\Xi_j(D_1C)+\Xi_j(C(A'+\epsilon\Delta')C)+\Xi_j(CD_2)\\
    &=\Xi_j(L)+\Xi_j(R)+\Xi_j(C(A'+\epsilon\Delta')C).
  \end{align*}

The equality between the second and third lines arises because the matrices $D_1C$,
$C(A'+\epsilon\Delta')C$ and $CD_2$ and their product have non-zero entries
only in the top left $j\times j$ submatrix. For such matrices, $\Xi_j(\cdot)$
is numerically equal to the logarithm of the 
absolute value of the determinant of the submatrix.
Since the determinant is multiplicative, the equality follows.

  Since Lebesgue measure on $\mathbf B=\{M\colon \|M\|\le 1\}$ is preserved by
  the operations of pre- and post-multiplying by an orthogonal matrix,
  it suffices to show that there exists $K>0$ such that 
  for any matrix $A$,
  \begin{equation}\label{eq:epsest1}
    \int_{\mathbf B}\Xi_j(C(A'+\epsilon\Delta)C)\,d\lambda(\Delta)\ge 
    K\log\epsilon\quad
    \text{for $\epsilon<\tfrac12$}.
  \end{equation}

  Let $A''$ be the top left $j\times j$ submatrix of $A'$ and notice
  that the measure on the top left $j\times j$ submatrix of $\Delta$
  is absolutely continuous with respect to the measure on $j\times j$
  matrices with uniform entries in $[-1,1]$ with bounded
  density. As noted above, $\Xi_j$ agrees with the logarithm of the
  absolute value of the determinant for a $j\times j$ matrix.

  Hence to establish \eqref{eq:epsest1}, it suffices to give a
  logarithmic lower bound:
  \begin{equation}\label{eq:epsest2}
    \int_{\mathbf U}\log\det(A''+\epsilon U)\,d\lambda'(U)\ge
    K\log\epsilon\quad\text{for $\epsilon<\tfrac12$},
  \end{equation}
  where $\mathbf U$ is the collection of $j\times j$ matrices with
  entries in $[-1,1]$ and $\lambda'$ is the uniform measure on $\mathbf
  U$. One checks, thinking of the columns of $U$ being generated one
  at a time, that the probability that the $i$th column lies within a
  $\delta$-neighbourhood of the span of the previous columns is at
  most $O(\delta/\epsilon)$, so the probability that the determinant of 
  $A''+\epsilon U$
  is less than $\delta^j$ is $O(j\delta/\epsilon)$. Hence we obtain
  \begin{equation*}
    \PP(-\log\det(A''+\epsilon U)>k)\le \min(1,Cje^{-k/j}/\epsilon).
  \end{equation*}
  Using the estimate for non-negative random variables $\mathbb EX\le
  \sum_{k=0}^\infty \PP(X\ge k)$, we obtain the bound $\mathbb
  E(-\log\det(A''+\epsilon U))\lesssim j|\log\epsilon|$.

  From this, we obtain the $O(|\log\epsilon|)$ bound as required.
\end{proof}

\begin{step}Putting it all together.
\end{step}

We apply this by grouping each consecutive string of good blocks into a
single matrix (and using \eqref{eq:goodblox}) and also grouping
strings of consecutive bad blocks minus the first and last matrices
into a single matrix (and using \eqref{eq:intbadblox}). The first and last
matrices of a string of bad blocks are then handled with Lemma
\ref{lem:trans}.

More specifically, we condition on $\omega\in\Omega$ and calculate
$\int\Xi_j(\cocycleps{MN}\omega)\,d\lambda^{MN}$.
Let $S=\{0\le i<M\colon \cocycle N{\sigma^{iN}\omega}\text{ is bad}\}$.
Let $r(\om)=|S|$ and $0<b_1<b_2<\ldots<b_r<M$ be the
increasing enumeration of $S$. Also let $b_0=-1$ and $b_{r+1}=M$.
Then we factorize $\cocycleps{MN}\om$ and $\cocycle {MN}\om$ as 
\begin{align*}
\cocycleps{MN}\omega&=\tilde G_r\tilde B_r\cdots 
\tilde B_2\tilde G_1\tilde B_1\tilde G_0\text{\,; and}\\
\cocycle{MN}\omega&= G_r B_r\cdots B_2 G_1 B_1 G_0,
\end{align*}
where $G_i=\cocycle{(b_{i+1}-b_i-1)N}{\sigma^{(b_i+1)N}\om}$,
$\tilde G_i=\cocycleps{(b_{i+1}-b_i-1)N}{\sigma^{(b_i+1)N}\om}$,
$B_i=\cocycle N{\sigma^{b_iN}\om}$ and
$\tilde B_i=\cocycleps N{\sigma^{b_iN}\om}$ (so the $G_i$ are 
products of consecutive good blocks and $B_i$ are (single) bad blocks).
We further factorize $B_i$ and $\tilde B_i$ as
$\tilde B_i=A^\ep_{\sigma^{(b_i+1)N-1}\om}\tilde C_iA^\ep_{\sigma^{b_iN}\om}$ and
$B_i=A_{\sigma^{(b_i+1)N-1}\om}C_iA_{\sigma^{b_iN}\om}$, where
$\tilde C_i=\cocycleps{N-2}{\sigma^{b_iN+1}\om}$ and
$C_i=\cocycle{N-2}{\sigma^{b_iN+1}\om}$.

Now using Lemma \ref{lem:trans} (and the constant $K$ from its statement), we have
\begin{equation*}
\int \Xi_j(\cocycleps{MN}\om)\,d\lambda^{MN}\ge
\int \left(\sum_{i=0}^{r(\om)}\Xi_j(\tilde G_i)+\sum_{i=1}^{r(\om)}\Xi_j(\tilde C_i)
\right)\,d\lambda^{MN}+r(\om)K|\log\ep|.
\end{equation*}

From \eqref{eq:goodblox}, we have $\sum_{i=0}^r \Xi_j(\tilde G_i)
\ge \sum_{i=0}^r \Xi_j(G_i)+Mj\log(D\delta/3)$ for all values of the
perturbation matrices that occur inside those blocks.
From \eqref{eq:intbadblox}, we have for each $1\le i\le r(\om)$, 
\begin{equation*}
\int\Xi_j(\tilde C_i)\,d\lambda^{N-2}(\Delta_{b_iN+1},\ldots,
\Delta_{(b_{i}+1)N-2})
\ge \Xi_j(C_i)-1.28d^2(N-2)j.
\end{equation*}
Letting $E(\om)=Mj\log(D\delta/3)-1.28d^2(N-2)jr(\om)+r(\om)K\log\ep$ and
combining the inequalities together with subadditivity of $\Xi_j$ yields
\begin{equation}\label{eq:cocbd}
\begin{split}
&\frac{1}{MN}\int\Xi_j(\cocycleps {MN}\om)\,d\Lambda^{MN}
\ge\frac{1}{MN}\sum_{i=0}^{r(\om)} \Xi_j(G_i)+\frac{1}{MN}
\sum_{i=1}^{r(\om)}\Xi_j(C_i)+\frac{E(\om)}{MN}\\
&\ge \frac{\Xi_j(\cocycle {MN}\om)}{MN}-\frac{1}{MN}\sum_{i=1}^{r(\om)} 
\left(\Xi_j^+(A_{\sigma^{b_iN}\om})+
\Xi_j^+(A_{\sigma^{(b_i+1)N-1}\om})\right)+\frac{E(\om)}{MN}.
\end{split}
\end{equation}

By \eqref{eq:etachoice} and \eqref{eq:Ncond}, we see 
$(1/MN)\int E(\om)\,d\bbp(\om)>-3\chi/4$.
Finally, we have
\begin{equation*}
\frac{1}{MN}\int\left(\sum_{i=1}^{r(\om)} (\Xi_j^+(A_{\sigma^{b_iN}\om})+
\Xi_j^+(A_{\sigma^{(b_i+1)N-1}\om}))\right)
\le \frac{2}{N}\int\Xi_j^+(A_\om)\,d\bbp(\om)<\frac\chi4.
\end{equation*}

Combining these inequalities, we obtain
\begin{equation*}
\frac1{MN}\int\Xi_j(\cocycleps{MN}\om)\,d\bar\bbp(\bar\om)
\ge \frac{1}{MN}\int\Xi_j(\cocycle{MN}\om)\,d\bbp(\om)-\chi.
\end{equation*}

Taking the limit as $M\to\infty$, we deduce
$\liminf_{\ep\to 0}(\mu_1^\ep+\ldots+\mu_j^\ep)
\ge (\mu_1+\ldots+\mu_j)-\chi$. 
Since $\chi>0$ was arbitrary, we deduce \eqref{eq:liminf}.

\end{proof}

\section{Convergence of Oseledets spaces}\label{sec:Spaces}

From Theorem \ref{thm:main}\eqref{part:expts}, we have established the existence
of an $\epsilon_0>0$ such that for $\epsilon<\epsilon_0$, in the perturbed
matrix cocycle,
$\mu^\ep_j\in(\lambda_i-\tau,\lambda_i+\tau)$ for all $j$ 
satisfying $D_{i-1}<j\le D_i$.
Recall that $Y_i^\ep(\oom)$ was defined to be the sum of the Oseledets spaces 
corresponding to exponents in the range $(\lambda_i-\tau,\lambda_i+\tau)$,
with $Y_i(\om)$ being the corresponding spaces for the unperturbed matrix cocycle.
Let $\topspaceDSunpert i\om=\bigoplus_{k\le i}Y_i(\om)$ be the fast subspace
for the unperturbed matrix cocycle and 
$\bottomspaceDSunpert i\om=\bigoplus_{k>i} Y_i(\om)$ be the slow subspace. 
We similarly introduce notation $\topspaceDSpert i\oom$ and 
$\bottomspaceDSpert i\oom$
in the perturbed matrix cocycle. Notice that $\oom=(\om,\Delta)$, and so
$\topspaceDSunpert i\om$ and $\topspaceDSpert i\oom$ may be regarded as living
on the same probability space $\bar\Omega$.

The proof of Theorem \ref{thm:main}\eqref{part:spaces} will follow 
relatively straightforwardly from the following lemma whose proof will
occupy this section.

\begin{lem}\label{lem:convSpaces}
Let $0<\chi<1$. 
Let $\topspaceDSunpert i\omega$
and $\topspaceDSpert i\oom$ be as above, 
corresponding to the largest $D_i$ Lyapunov exponents of the unperturbed and perturbed cocycles, respectively.
Then, for every $\ep$ sufficiently small,
\[
 \bar\bbp \big(\oom:\angle(\topspaceDSpert i{\oom}, 
 \topspaceDSunpert i\om)>\chi \big)<\chi.
\]
In particular,
$\topspaceDSpert i{\oom}$ converges in probability to 
$ \topspaceDSunpert i\om$ as $\ep\to 0$. 

\end{lem}
Recalling that $\bottomspaceDSunpert i\om=
\big( \topspacestarDSunpert i\om \big)^\perp$,
where $\topspacestarDSunpert i\om$ 
denotes the Oseledets space of the cocycle dual to $A$, 
Lemma~\ref{lem:convSpaces} immediately implies the following.

\begin{cor}
\label{cor:convSlowSpaces}
Let $\bottomspaceDSunpert i\omega$
and $\bottomspaceDSpert i\oom$ denote 
the slow Oseledets subspaces of the unperturbed and perturbed cocycles, 
respectively, as described above.
Then $\bottomspaceDSpert i{\oom}$ converges in probability to 
$ \bottomspaceDSunpert i\om$ as $\ep\to 0$.
\end{cor}

\begin{proof}[Proof of Theorem \ref{thm:main}(\ref{part:spaces})
from Lemma \ref{lem:convSpaces} ]
Notice that $Y_i(\om)=\topspaceDSunpert i\om\cap
\bottomspaceDSunpert{i-1}\om$ and 
$Y_i^\ep(\oom)=\topspaceDSpert i\oom\cap
\bottomspaceDSpert{i-1}\oom$,
so we want to show that $\bottomspaceDSpert {i-1}\oom\cap 
\topspaceDSpert i\oom$ converges
in probability to $\bottomspaceDSunpert{i-1}\omega\cap 
\topspaceDSunpert i\omega$
as $\epsilon\to 0$. We also have
\begin{align*}
\bottomspaceDSpert {i-1}\oom\cap \topspaceDSpert i\oom&=
\operatorname{Pr}_{\topspaceDSpert i\oom \parallel 
\bottomspaceDSpert i\oom}
(\bottomspaceDSpert {i-1}\oom);\text{ and}\\
\bottomspaceDSunpert {i-1}\omega\cap \topspaceDSunpert i\omega&=
\operatorname{Pr}_{\topspaceDSunpert i\omega \parallel 
\bottomspaceDSunpert i\omega}
(\bottomspaceDSunpert {i-1}\omega).
\end{align*}
	
Now, lemma 6  of \cite{FLQ2}, together with Lemma \ref{lem:convSpaces} and
the separation of $\topspaceDSunpert i\omega$ and $\bottomspaceDSunpert i\omega$ 
guaranteed by Lemma~\ref{lem:goodblocks} do the job.
\end{proof}
	
\subsection{Strategy and notation}
Throughout, we shall let $j=D_i$, so that we are studying evolution of
$j$-dimensional subspaces.
In order to show Lemma~\ref{lem:convSpaces},
we will assume that all of the perturbations $(\Delta_n)$ are fixed 
except for the $-1$ time
coordinate. That is, we compute the probability that the perturbed
and unperturbed fast spaces are close conditioned on $(\Delta_n)_{n\ne-1}$ and 
$\omega$.

We think of 
$\topspaceDSpert j\oom$ as a random variable (depending on $\Delta_{-1}$),
then applying the sequence of
matrices (all already fixed), $\cocycleps{nN}\oom$, show that the resulting
$j$-dimensional subspace is highly likely to be closely aligned to
$\topspaceDSunpert j{\sigma^{nN}\om}$.

To control the evolution, we successively apply $\cocycleps N\oom$,
$\cocycleps N{\ssig^N\oom},\ldots, \cocycleps N{\ssig^{(n-1)N\oom}}$, 
where $\oom=(\om, \Delta)$, $s$ denotes the left shift and  
$\ssig(\oom)=(\sig \om, s \Delta)$.
We will assume that the underlying blocks of unperturbed $A$'s are 
good blocks. The number, $n$, of
steps will be fixed. In fact, $n$ will
depend on the difference $\lambda_j-\lambda_{j+1}$ and the quantity, 
$C$, appearing in Lemma~\ref{lem:Clogeps}.
Hence for small $\eta$, it will be very
likely that one has $n$ consecutive good blocks.

We shall use the following parameterization of the Grassmannian of 
$j$-dimensional subspaces of $\R^d$. Let $f_1,\ldots,f_j$ be a basis
for $F$, a $j$-dimensional subspace, and $e_1,\ldots,e_{d-j}$ a basis
for $E$, a complementary subspace. Now for any $j$-dimensional vector space $V$
with the property that $V\cap E=\{0\}$, each $f_k$ can be uniquely expressed in the
form $v_k-\sum_i b_{ik}e_i$ where $v_k\in V$. The parameterization of $V$ with 
respect to $(F,E)$ chart (or more formally with respect to chart arising from
$F$ and $E$ with their chosen bases)
is the matrix $B=(b_{ik})_{1\le i\le d-j,1\le k\le j}$. Conversely, given the matrix
$B$, one can easily recover a basis for $V$: $v_k=f_k+\sum_i b_{ik}e_i$ and hence
the subspace $V$.

\begin{lem}\label{lem:perpcalc}
Let $F$ and $E$ be orthogonal complements in $\R^d$ and let $(f_i)_{1\le i\le j}$
and $(e_i)_{1\le i\le d-j}$ be orthonormal bases. Let $V$ be a $j$-dimensional
subspace of $\R^d$ such that $V\cap E=\{0\}$. Let the 
parameterization of $V$ be $B$. Then
\begin{equation*}
\perp(V,E)=\sqrt{1-\frac{\|B\|}{\sqrt{1+\|B\|^2}}}.
\end{equation*}
In particular for any $M>1$, $\|B\|\le M$ implies $\perp(V,E)\ge 1/(2M)$.
\end{lem}

\begin{proof}
	Let $v_k=f_k+\sum_i b_{ik}e_i$ so that $(v_k)$ forms a basis for $V$. Now
	let $v=\sum_k c_kv_k=\sum c_kf_k +\sum_i(Bc)_ie_i$ belong to $V\cap S$, so that
	$\|c\|^2+\|Bc\|^2=1$. The closest point in $E\cap S$ to $v$
	is $(1/\|Bc\|)\sum_i(Bc)_ie_i$, 
	which is at a square distance $\|c\|^2+(\|Bc\|-1)^2=2(1-\|Bc\|)$ from $v$. 
	This distance is minimized when $c$ is the multiple of the 
	dominant singular vector of $B$
	for which $\|c\|^2+\|Bc\|^2=1$. That is, 
	$\|c\|=1/\sqrt{\|B\|^2+1}$ and $\|Bc\|=\|B\|/
	\sqrt{\|B\|^2+1}$. Substituting this, we obtain the claimed formula for $\perp(V,E)$.
	
\end{proof}

We will do the iteration using the following steps:
\begin{enumerate}[(S1)]
\setcounter{enumi}{-1}
\item Express $V=\topspaceDSpert j{\oom}$
as a matrix $B$ using the $(\bottomspacemat j{\cocycleps N{\oom}}^\perp,
\bottomspacemat j{\cocycleps N{\oom}})$ chart.\label{step:newzero}
Set $i=0$.\item Let $C_i=\cocycleps N{\ssig^{iN}\oom}$. Compute $C_i (V)$ in the
  $(\topspacemat j{C_i}  ,\topspacemat j{C_i}^\perp)$ chart; this is 
  straightforward as $C_i$ is diagonal with respect to the pair of
  bases on the domain and range spaces.\label{step:newone}
\item Change bases to the $(\bottomspacemat j{C_{i+1}}^\perp,
\bottomspacemat j{C_{i+1}})$ chart.
Update $V$ to $C_i(V)$,
increase $i$ and repeat steps 1 and 2 a total of $n$ times.\label{step:newtwo}
\end{enumerate}

We will see that above transformations are given by fractional linear
transformations on matrices, and use this formalism, together
with properties of multivariate normal distributions to
establish the fact that the fast Oseledets space for the perturbed 
matrix cocycle is with high probability close to the fast Oseledets 
space for the unperturbed matrix cocycle.

\begin{proof}[Proof of Lemma \ref{lem:convSpaces}]
Recall that, 
once the initial cocycle $A_\om$ and $j$ are fixed, the 
Lyapunov exponents $\lam_j, \lam_{j+1}$ 
and the constant $C$ 
appearing in Lemma \ref{lem:Clogeps},
are fixed as well. Fix $n$ satisfying
\begin{equation}\label{eq:nchoice}
n(\lam_j-\lam_{j+1}-6\tau)>1/C.
\end{equation}

We apply Lemma \ref{lem:goodblocks} with $\eta_0=\chi/(2n+2)$ and
$\delta_1=\chi/2$. Let $\kappa$, $\delta$, $G$ and $n_0$ be as in the conclusion 
of the lemma. 

We now fix the range of $\epsilon$ in which we will obtain the required 
closeness of the top spaces. We shall set $N=\lfloor C |\log\epsilon|\rfloor$,
and will require that $N$ be large enough (and hence
that $\epsilon$ should be small enough) to simultaneously 
satisfy a number of conditions:

\begin{enumerate}
\renewcommand{\theenumi}{C\arabic{enumi}} 
 \renewcommand{\labelenumi}{(\theenumi)}
\item $N > n_0$;\label{cond:goodblocks}
\item $e^{N\tau}>8/\delta^2$ (and since $\del<2$, $e^{N\tau}>1+2/\delta$);
\label{cond:8del}
\item $\exp(N(n(\lambda_j-\lambda_{j+1}-6\tau)-1/C))>4(e\pi/2)^{d^2/2}j^{3/2}/\chi$;
\label{cond:eps}
\item $\PP(\|A_\om\|>e^{\tau nN}-1)<\chi/4$;\label{cond:prob}
\item $N>\log(2/\delta)/\tau$;\label{cond:dunno}
\item $e^{N(-\lambda_j+\lambda_{j+1}+2\tau)}<5\del^2/(1+\del)$\label{cond:del2}.
\end{enumerate}

Let $\tld{G}=\bigcap_{j=0}^{n} \sig^{-jN}G$.
Then $\bbp (\tld{G})\geq 1-\chi/2$. Furthermore, we have the following result, whose
proof is deferred until \S\ref{sec:pfClaim}.
\begin{claim}\label{claim:conv4good}
Assume $N$ satisfies conditions \eqref{cond:goodblocks}--\eqref{cond:dunno}. Then,
  \begin{equation}\label{eq:toprove}
\bar \PP(\angle\big(\topspaceDSpert j{\oom},
 \topspaceDSunpert j\om\big)>\chi|\om)\le\chi/2\text{ for all $\om\in\tilde G$.}
 \end{equation}
\end{claim}
With this result at hand, the proof of Lemma~\ref{lem:convSpaces} may be immediately concluded, as follows.
\begin{proof}[Proof of Lemma~\ref{lem:convSpaces} using Claim~\ref{claim:conv4good}]
\begin{align*}
 \bar\bbp \big( \oom &:\angle(\topspaceDSpert j{\oom}, 
 \topspaceDSunpert j\om) >\chi \big) \\
 &\leq \bar\bbp(\angle(\topspaceDSpert j{\oom},\topspaceDSunpert j\om)
>\chi \ |\ \om \in \tld{G})  \cdot \bbp(\tld{G}) + (1-\bbp(\tld{G}))\\
 &\leq \chi/2 \cdot(1-\chi/2)+\chi/2 < \chi.
\end{align*}
\end{proof}

\subsection{Proof of Claim~\ref{claim:conv4good}}\label{sec:pfClaim}

Let $V_0=\topspaceDSpert j{\oom}$, which we consider to be a 
random variable by fixing all matrices
except the $-1$st, and let $V_{i+1}=C_i(V_i)$. 
Write $B_i$ for the matrix of $V_i$ with respect to the 
$(\bottomspacemat j{C_i}^\perp,\bottomspacemat j{C_i})$ basis, 
as explained in Step~\ref{step:newzero} above;
so that in particular, $B_0$ is a random variable with $\epsilon$-variability.

Let $R_i$ be the matrix describing multiplication by $C_i$ with respect to the 
$(\bottomspacemat j{C_i}^\perp,\bottomspacemat j{C_i})$ and
$(\topspacemat j{C_i},\topspacemat j{C_i}^\perp)$ bases. 
This corresponds to Step~\ref{step:newone}.
Let $P_i$ (corresponding to Step \ref{step:newtwo}) be the basis change matrix
from the $(\topspacemat j{C_i},\topspacemat j{C_i}^\perp)$ to the 
$(\bottomspacemat j{C_{i+1}}^\perp,\bottomspacemat j{C_{i+1}})$ basis.

Then, $R_i$ is diagonal, say
\begin{align*}
R_i&:=\begin{pmatrix}D_{2,i}&0\\0&D_{1,i}\end{pmatrix},
\end{align*}
where $D_{1,i}$ is the diagonal matrix with
entries $s_{j+1},\ldots,s_d$ and $D_{2,i}$ is the diagonal matrix with
entries $s_1,\ldots,s_j$, where $s_1\geq \ldots \geq s_d$ are
the singular values of $C_i$. Notice that the ratio between 
the largest entry of
$D_{2,i}$ and the smallest entry of $D_{1,i}$ is at least
$e^{(\lambda_j-\lambda_{j+1}-\tau)N}$.

Notice that $P_i$ is an orthogonal matrix, as it is
the change of basis matrix between two orthonormal bases. Let
\begin{align}
P_i&:=\begin{pmatrix}\zeta_i&\gamma_i\\\beta_i&\alpha_i\end{pmatrix},
\text{ and }\\
Q_i&:=P_iR_i=:\begin{pmatrix}q_i&s_i\\p_i&r_i\end{pmatrix}. \label{eq:Qi}
\end{align}
We use a similar argument to that in Lemma \ref{lem:perpcalc} to estimate 
$\|\zeta^{-1}\|$. From the definition of the good set $G$ (after \eqref{eq:Ncond}), we know
$\perp(\topspacemat j{C_i},\bottomspacemat j{C_{i+1}})>6\delta$.
Let $\topspacemat j{C_i}$ be spanned by the singular vector images
$f_1,\ldots,f_j$; $\bottomspacemat j{C_{i+1}}^\perp$ be spanned by
the singular vectors $g_1,\ldots,g_j$ and $\bottomspacemat j{C_{i+1}}$ be
spanned by $h_1,\ldots,h_{d-j}$. 
In particular, if $a_1^2+\ldots+a_j^2=1$ and $v=a_1f_1+\ldots+a_jf_j$, 
then with respect to the $((g_k),(h_k))$ basis, $v$ has coordinates
$(\zeta_i a,\beta_i a)$. The nearest point in the unit sphere of
$\bottomspacemat j{C_{i+1}}$ 
has coordinates $\beta_i a/\|\beta_i a\|$ with respect to the $(h_k)$ vectors.
The distance squared between the two points is, by the calculation in
Lemma \ref{lem:perpcalc}, $2-2\|\beta_i a\|$. By the goodness property, this exceeds
$72\delta^2$, so that $1-\|\beta_ia\|\ge 36\delta^2$ and $\|\zeta_ia\|^2
=(1-\|\beta_ia\|)(1+\|\beta_ia\|)\ge 36\delta^2$. In particular, we deduce
\begin{equation}\label{eq:zetabd}
\|\zeta_i^{-1}\|\ge 1/(6\delta).
\end{equation}

Let us also note that the $p_i$, $q_i$, $r_i$ and
$s_i$ depend only on the choice of matrices from time 0 onwards and hence have
been fixed by the conditioning, whereas $B_0$ is a random quantity whose 
conditional distribution we will study in \S\ref{sec:DistB0}.

Notice that the matrix $Q_i$ is characterized by the property 
that if the coordinates of $x\in\R^d$ with
respect to the $(\bottomspacemat j{C_i}^\perp,\bottomspacemat j{C_i})$
basis are given by $z$, then the coordinates of
$\cocycleps N{\ssig^{iN}\oom}x$ are given by $Q_iz$ with respect to the 
$(\bottomspacemat j{C_{i+1}}^\perp, \bottomspacemat j{C_{i+1}})$ basis. 

Set $F_0=I$, $H_0=B_0$. 
Let
$$
\begin{pmatrix}F_i\\H_i\end{pmatrix}=
Q_{i-1}\ldots Q_1Q_0\begin{pmatrix}F_0\\H_0\end{pmatrix}.
$$ 
Recall that $B_i$ is the matrix of $V_i$ with respect to the 
$(\bottomspacemat j{C_i}^\perp,\bottomspacemat j{C_i})$ basis.
Then, we have 
$$
B_i=H_iF_i^{-1}.
$$
To see this, consider a point $x$ of $V$ expressed in terms of the 
$(E(C_0)^\perp,E(C_0))$ basis 
as $z$. Then with respect to the $(\bottomspacemat j{C_{i}}^\perp,
\bottomspacemat j{C_{i}})$ basis, $\cocycleps {iN}\oom x$ has coordinates
$Q_{i-1}\ldots Q_0z$.  $\cocycleps {iN}{\oom} V$ has a basis 
expressed in coordinates of the 
$(\bottomspacemat j{C_i}^\perp,\bottomspacemat j{C_i})$
basis given by the columns of $\begin{pmatrix}F_i\\H_i\end{pmatrix}$.
Post-multiplying by $F_i^{-1}$ gives an alternative basis for 
$\cocycleps {iN}\oom$ expressed in terms 
of the $(\bottomspacemat j{C_i}^\perp,\bottomspacemat j{C_i})$ basis, 
as the columns of 
$\begin{pmatrix}I\\H_iF_i^{-1}\end{pmatrix}$ as required.

In view of Lemmas \ref{lem:LY} and \ref{lem:perpcalc}, 
it remains to show that $B_n=H_nF_n^{-1}$ is controlled for most
choices of the perturbation $\Delta_{-1}$. 

Let
$$
Q_{n-1}\ldots Q_0=
\begin{pmatrix}
W^{(n)}&X^{(n)}\\
Y^{(n)}&Z^{(n)}
\end{pmatrix},
$$
where the $W$'s are $j\times j$, $X$'s are $j\times (d-j)$, $Y$'s are
$(d-j)\times j$ and $Z$'s are $(d-j)\times (d-j)$.

\subsubsection{Singular values and invertibility of $W^{(n)}$}\label{sec:SingValWn}
Let us now show that all $j$ singular values of $W^{(n)}$ are bounded below by a quantity 
close to  $e^{nN\lambda_j}$. This will immediately imply invertibility of $W^{(n)}$. 
We start by proving $\|W^{(k)}x\|\ge \delta\|Y^{(k)}x\|$ for all $x$ and $k\geq 0$.

By (\ref{cond:del2}), 
if $c>\delta$, then
$$
\frac{6\delta c-e^{(\lambda_{j+1}-\lambda_j+2\tau)N}}
{c+e^{(\lambda_{j+1}-\lambda_j+2\tau)N}}>\delta.
$$

Notice that 
\begin{equation*}
\begin{pmatrix}
W^{(k+1)}\\Y^{(k+1)}
\end{pmatrix}
=\begin{pmatrix}q_{k}&s_{k}\\p_{k}&r_{k}\end{pmatrix}
\begin{pmatrix}W^{(k)}\\Y^{(k)}\end{pmatrix}.
\end{equation*}

Suppose $\|W^{(k)}x\|>c\|Y^{(k)}x\|$. Then we have
\begin{align*}
\|W^{(k+1)}x\|&\ge \|q_kW^{(k)}x\|-\|s_kY^{(k)}x\|\\
&\ge\|\zeta_k D_{2,k}W^{(k)}x\|-\|s_k\|\|Y^{(k)}x\|\\
&\ge6\delta \|D_{2,k}W^{(k)}x\|-\|D_{1,k}\|\|Y^{(k)}x\|,
\end{align*}
where we used \eqref{eq:zetabd} to obtain the third inequality. Similarly
\begin{align*}
\|Y^{(k+1)}x\|&\le \|p_kW^{(k)}x\|+\|r_kY^{(k)}x\| \\
&\le \|\beta_kD_{2,k}W^{(k)}x\|+\|r_k\|\|Y^{(k)}x\|\\
&\le \|D_{2,k}W^{(k)}x\|+\|D_{1,k}\|\|Y^{(k)}x\|.
\end{align*}
Using the facts that $\|D_{1,k}\|\le \exp(N(\lambda_{j+1}+\tau))$,
and $\|D_{2,k}x\|\ge \exp(N(\lambda_j-\tau))\|x\|$ for all $x$,
and assuming that $\|W^{(k)}x\|\ge c_{k}\|Y^{(k)}x\|$, we see
that $\|W^{(k+1)}x\|\ge c_{k+1}\|Y^{(k+1)}x\|$, where
$$
c_{k+1} = \frac{6\delta c_{k} e^{(\lambda_j-\tau)N}-e^{(\lambda_{j+1}+\tau)N}}
{e^{(\lambda_j-\tau)N}c_{k}+e^{(\lambda_{j+1}+\tau)N}}.
$$
Everything has been set up so that if $c_{k}>\delta$, then $c_{k+1}>\delta$.
It is easy to check that $c_0>\delta$ ($W^{(0)}=I$ and $Y^{(0)}=0$)
so we deduce that 
\begin{equation}\label{eq:W>Y}
\|W^{(k)}x\|\ge \delta\|Y^{(k)}x\|,  
\end{equation}
for all $x$ and $k$ as required.
From this, we see that 
\begin{align*}
  \|W^{(n)}x\|&\ge \delta e^{(\lambda_j-\tau)N}\|W^{(n-1)}x\|
  -e^{(\lambda_{j+1}+\tau)N}\|Y^{(n-1)}x\|\\
  &\ge (\delta
  e^{(\lambda_j-\tau)N}-(1/\delta)e^{(\lambda_{j+1}+\tau)N})\|W^{(n-1)}x\|\\
  &\ge e^{(\lambda_j-2\tau)N}\|W^{(n-1)}x\|,
\end{align*}
where for the last inequality, we used (\ref{cond:dunno}).

In particular we deduce inductively that if $x\ne 0$, then
$W^{(n)}x\ne 0$ for all $n$.  Thus, $W^{(n)}$ is invertible
and we see
\begin{equation}
s_j(W^{(n)})\geq e^{nN(\lambda_j -2\tau)}.\label{eq:sjWnbd}
\end{equation}

\subsubsection{Recursion for $E_n$}\label{sec:BoundEn}
Let $E_n=Z^{(n)}-Y^{(n)}{W^{(n)}}^{-1}X^{(n)}$. This matrix will play a role in bounding $B_n$.
Notice that
\begin{equation}\label{eq:En}
\begin{pmatrix}
W^{(n)}&X^{(n)}\\
Y^{(n)}&Z^{(n)}
\end{pmatrix}
\begin{pmatrix}
-{W^{(n)}}^{-1}X^{(n)}\\I\end{pmatrix}
=
\begin{pmatrix}
0\\E_n
\end{pmatrix}.
\end{equation}
In fact, $E_n$ may be defined this way: $E_n$ is the unique lower
submatrix $M$ such that there exists $A$ satisfying
\begin{equation*}
\begin{pmatrix}
W^{(n)}&X^{(n)}\\
Y^{(n)}&Z^{(n)}
\end{pmatrix}
\begin{pmatrix}
A\\I\end{pmatrix}
=
\begin{pmatrix}
0\\M
\end{pmatrix}.
\end{equation*}

Now we have 
\begin{equation*}
\begin{split}
\begin{pmatrix}
W^{(n+1)}&X^{(n+1)}\\
Y^{(n+1)}&Z^{(n+1)}
\end{pmatrix}
\begin{pmatrix}
-{W^{(n)}}^{-1}X^{(n)}\\I\end{pmatrix}
&=
\begin{pmatrix}
q_{n}&s_{n}\\p_{n}&r_{n}
\end{pmatrix}
\begin{pmatrix}
W^{(n)}&X^{(n)}\\
Y^{(n)}&Z^{(n)}
\end{pmatrix}
\begin{pmatrix}
-{W^{(n)}}^{-1}X^{(n)}\\I\end{pmatrix}\\
&=
\begin{pmatrix}
q_{n}&s_{n}\\p_{n}&r_{n}
\end{pmatrix}
\begin{pmatrix}
0\\E_n
\end{pmatrix}\\
&=
\begin{pmatrix}
s_{n}E_n\\r_{n}E_n
\end{pmatrix}.
\end{split}
\end{equation*}
To finalise the recursion setup,
we now seek matrices $B$  and $C$ such that
\begin{equation}\label{eq:Ecalcpart2}
\begin{pmatrix}
W^{(n+1)}&X^{(n+1)}\\
Y^{(n+1)}&Z^{(n+1)}
\end{pmatrix}
\begin{pmatrix}
B\\0\end{pmatrix}
=
\begin{pmatrix}
-s_{n}E_n\\C
\end{pmatrix}.
\end{equation}

Combining the above, we see
\begin{equation*}
\begin{pmatrix}
W^{(n+1)}&X^{(n+1)}\\
Y^{(n+1)}&Z^{(n+1)}
\end{pmatrix}
\begin{pmatrix}
B-{W^{(n)}}^{-1}X^{(n)}\\I\end{pmatrix}
=
\begin{pmatrix}
0\\r_{n}E_n+C
\end{pmatrix},
\end{equation*}
so that $E_{n+1}=r_{n}E_n+C$.

From \eqref{eq:Ecalcpart2}, we see that $B=-{W^{(n+1)}}^{-1}s_{n}E_n$, so
that $C=Y^{(n+1)}B=-Y^{(n+1)}{W^{(n+1)}}^{-1}s_{n}E_n$.  In
particular, we obtain
\begin{equation*}
E_{n+1}=\left(r_{n}-Y^{(n+1)}{W^{(n+1)}}^{-1}s_{n}\right)E_n.
\end{equation*}
Substituting $x={W^{(k)}}^{-1}z$ in \eqref{eq:W>Y},
we obtain $\|z\|\ge \delta/2\|Y^{(k)}{W^{(k)}}^{-1}z\|$, so that
$\{\|Y^{(k)}{W^{(k)}}^{-1}\|\}_{k\in \N}$ is uniformly bounded by $2/\delta$. 
Furthermore, from the definition of $Q_k$, \eqref{eq:Qi}, and the 
choice of $N$, $\|r_{k}\|,\|s_k\|\leq e^{N(\lambda_{j+1}+\tau)}$,
so that $\|E_{n+1}\|\le e^{N(\lambda_{j+1}+\tau)}(1+\frac2\delta)\|E_n\|$.
Hence by (\ref{cond:8del}), we obtain
\begin{equation}\label{eq:Enbounds}
\|E_n\|\leq e^{N n(\lambda_{j+1}+2\tau)}.
\end{equation}
Finally, using \eqref{eq:En}, we have that
\begin{equation}
\begin{split} \label{eq:B_nE_n}
B_n&=(Y^{(n)}+Z^{(n)}B_0)(W^{(n)}+X^{(n)}B_0)^{-1}\\
&=(Y^{(n)}+(Y^{(n)}{W^{(n)}}^{-1}X^{(n)}+E_n)B_0)(W^{(n)}+X^{(n)}B_0)^{-1}\\
&=Y^{(n)}(I+{W^{(n)}}^{-1}X^{(n)}B_0)(W^{(n)}+X^{(n)}B_0)^{-1}
+
E_nB_0(W^{(n)}+X^{(n)}B_0)^{-1}\\
&=Y^{(n)}{W^{(n)}}^{-1}+E_nB_0(W^{(n)}+X^{(n)}B_0)^{-1}.
\end{split}
\end{equation}

\subsubsection{Distribution of $B_0$}\label{sec:DistB0}

Recall that we conditioned on $\omega$ and $(\Delta_n)_{n\ne-1}$. This
determines the top subspace at time $-1$, as well as the
$\bottomspacemat j{\cocycleps N{\bar\sig^{kN}\oom}}$
and $\bottomspacemat j{\cocycleps N{\bar\sig^{kN}\oom}}^\perp$
spaces for each $k\ge 0$.

Let $V$ be a $d\times j$ matrix whose columns consist of an orthonormal
basis in $\R^d$ for
the fast space at time $-1$.  The fast space at time 0 has a basis
given by the columns of $(A_{\sigma^{-1}\om}+\epsilon\Delta_{-1})V$
(recall that $\Delta$ was assumed
to be independent of the other perturbations of
$(A(\sigma^n\omega))_{n\in\Z\setminus \{-1\}}$ that have already been
fixed). For this section, we write $A$ in place of $A_{\sig^{-1}\om}$ and 
$\Delta$ in place of $\Delta_{-1}$.
The coordinates of $(A+\epsilon\Delta)V$ in terms of the
$(\bottomspacemat j{\cocycleps N{\bar\sig^{kN}\oom}}^\perp,
\bottomspacemat j{\cocycleps N{\bar\sig^{kN}\oom}})$
basis are given by
\begin{equation*}
\begin{pmatrix}F^T\\
E^T\end{pmatrix}
(A+\epsilon\Delta)V
=:
\begin{pmatrix}Z_1\\Z_2\end{pmatrix}
,
\end{equation*}
where $F$ is the matrix whose column vectors are the (orthonormal)
basis for $E^\perp$ and $E$ is the matrix whose column vectors are the
orthonormal basis for $E$. Specifically, the $j$th column of this
matrix gives the $(E^\perp,E)$ coordinates of the image of the $j$th
basis vector of $V$ under $A+\epsilon\Delta$. The matrix $B_0$ is then
given by $Z_2Z_1^{-1}$, that is $(E^T(A+\epsilon\Delta)
V)(F^T(A+\epsilon \Delta) V)^{-1}$.

\subsubsection{Bounds on $B_n$}
Substituting the expression for $B_0$ into \eqref{eq:B_nE_n}, we get
\begin{align*}
B_n&=Y^{(n)}{W^{(n)}}^{-1}+E_nE^T(A+\epsilon\Delta)
V(F^T(A+\epsilon \Delta) V)^{-1}\,\cdot\\
&\qquad\qquad\qquad\qquad\left[W^{(n)}+X^{(n)}
E^T(A+\epsilon\Delta) 
V(F^T(A+\epsilon \Delta) V)^{-1}\right]^{-1}\\
&=Y^{(n)}{W^{(n)}}^{-1}+E_nE^T(A+\epsilon\Delta) V
\left[W^{(n)}F^T(A+\epsilon\Delta)V+X^{(n)}E^T(A+\epsilon\Delta)V\right]^{-1}\\
&=Y^{(n)}{W^{(n)}}^{-1}+E_nE^T V(A+\ep\Delta)(UAV+\ep U\Delta V)^{-1},
\end{align*}
where $U=W^{(n)}F^T+X^{(n)}E^T$.

We make the following definitions:
\begin{equation}\label{def:MD}
\begin{split}
M&=E_nE^T(A+\epsilon\Delta) V\text{;}\\
D&=UAV\text{;\qquad\qquad\qquad and}\\
\tilde \Delta&=D+\epsilon U\Delta V.
\end{split}
\end{equation}
Now we have
\begin{equation}\label{eq:BnFinal}
B_n=Y^{(n)}{W^{(n)}}^{-1}+M
\tilde\Delta^{-1}.
\end{equation}

What remains is to give an upper bound on $\|M\|$ and
to show that 
$\|\tilde \Delta^{-1}\|$ is small for a large set of $\Delta$'s.

\subsubsection{Bounds on $\|\tilde \Delta^{-1}\|$ using multivariate 
normal random variables}\label{sec:NormalBounds}
 
\subsubsection*{Reduction to normal case}
We want bounds on $\mathbb P(\|(D+\epsilon U\Delta V)^{-1}\|\ge T)$.
We obtain these by comparing with the corresponding probability
when $\Delta$ is replaced by a matrix of independent 
standard normal random variables.
Let $Z$ be a $d\times d$ matrix of independent standard normal random
variables. We first show it suffices to compute 
$\mathbb P(\|(D+\epsilon UZV)^{-1}\|\geq T)$.

For a $T>0$ to be fixed below, let $R$ be the subset of
$d\times d$ matrices $C$ with entries in $[-1,1]$ such that
$\|(D+\epsilon UCV)^{-1}\|\geq T$.

Notice that $\mathbb P(Z \in R)=\int_R f_Z(X)\,dX$, where
$f_Z(X)=(2\pi)^{-d^2/2}\exp(-\sum_{1\le i,j\le d}X_{ij}^2/2)$ is the
density function of the $d\times d$ matrices with $N(0,1)$ entries. In
particular $f_Z(X)\ge (2\pi e)^{-d^2/2}$ for all matrices $X$ with
entries in $[-1,1]$, so that $\mathbb P(Z\in R)\ge (2\pi
e)^{-d^2/2}\text{Vol}(R)$ (where $\text{Vol}(R)$ is the volume of $R$
as a subset of $\R^{d^2}$). Similarly $\mathbb P(\Delta\in R)=
2^{-d^2}\text{Vol}(R)$ since the probability density function of $d\times d$
matrices taking values in $[-1,1]$ is $2^{-d^2}$. We see that 

\begin{equation}\label{eq:Nunifcomp}
\mathbb
P(\Delta\in R)\le (e\pi/2)^{d^2/2}\mathbb P(Z\in R). 
\end{equation}

\subsubsection*{Bound in the normal case}
Now let $\tilde Z=D+\epsilon UZV$ and let $\hat
V_i=\text{span}\{c_1,\ldots,c_{i-1},c_{i+1},\ldots,c_j\}$, where the
$c_i$ are the columns of $\tilde Z$.  Let $d_i=d(c_i,\hat
V_i)$. Let $x=(x_1,\ldots,x_j)^T$. Now we have $\|\tilde Z
x\|=\|x_ic_i+(x_1c_1+\ldots+x_{i-1}c_{i-1}+x_{i+1}c_{i+1}+\ldots+
x_jc_j)\|\ge d(x_ic_i,\hat V_i)=|x_i|d_i$. We see that $\|\tilde Z
x\|\ge \max_i (|x_i|d_i)\ge (\min_i d_i)\max_i|x_i|\ge \|x\|\min_i
d_i/\sqrt j$. In particular, we deduce that $\|\tilde Z^{-1}\|\le
\sqrt j/\min_i d_i$.
We now have 
\begin{equation}\label{eq:colsep}
\mathbb P(\|\tilde Z^{-1}\|\ge T)\le \mathbb P(\min_i
d_i\le \sqrt j/T)\le\sum_i \mathbb P(d_i\le \sqrt j/T).
\end{equation}

Notice that the entries of the matrix $D+\epsilon UZV$ 
have a multivariate normal distribution.
Any two such distributions with the same means and covariances are 
identically distributed. Recall that $V$ is a $d\times j$ matrix whose 
columns are pairwise orthogonal. A
consequence of this is that $ZV$ has the same distribution as a 
$d\times j$ matrix of independent standard normal
random variables. To see this, we see immediately that the expectation of
each entry is 0. We then need
to check the covariances, recalling that the columns of $V$ are orthonormal, 
we get:

\begin{align*}
\cov((ZV)_{ab},(ZV)_{cd})&=
\sum_{l,m} \cov(Z_{al}V_{lb},Z_{cm}V_{md})\\
&=\sum_{l,m}V_{lb}V_{md}\cov(Z_{al},Z_{cm})\\
&=\sum_{l,m}V_{lb}V_{md}\delta_{ac}\delta_{lm}\\
&=\delta_{ac}\sum_m V^T_{bm}V_{md}=\delta_{ac}(V^TV)_{bd}=\delta_{ac}\delta_{bd},
\end{align*}
as required.  

Let $Z'=ZV$. We next observe (by an identical calculation) that
$F^TZ'$ and $E^TZ'$ are distributed
as independent $j\times j$ and $(d-j)\times j$ matrices
with independent standard normal entries. 
Let $Z_1=F^TZV$ and $Z_2=E^TZV$. Recall that $\tilde Z=
D+\ep UZV$ and $U=W^{(n)}F^T+X^{(n)}E^T$. By 
\eqref{eq:colsep}, we are interested in the columns of 
$\tilde Z=D+\epsilon W^{(n)}Z_1
+\epsilon X^{(n)}Z_2$.

For a fixed $i$, we compute the probability that the distance 
of the $i$th column of $\tilde Z$
is distant at least $\sqrt j/T$ from the span of the other columns. 
We give a uniform estimate 
on this probability conditioned on the columns of $Z_1$ other than the $i$th
and the value of $Z_2$.
Having fixed all of this data,
let $\mathbf n$ be a unit normal vector to the $(j-1)$-dimensional space
spanned by the other columns (a constant given the data).
We then want to estimate $\mathbb P\left(|\mathbf n\cdot
(D^{(i)}+\epsilon X^{(n)}Z_2^{(i)}+\epsilon W^{(n)}Z_1^{(i)})|<\sqrt j/T\right)$, 
where the superscript $(i)$ 
indicates we are considering the $i$th column. 

Let $A=\mathbf n\cdot (D^{(i)}+\epsilon X^{(n)}Z_2^{(i)})$ and 
$\mathbf v=\epsilon \mathbf n^TW^{(n)}$ (both are constant given the data
on which we conditioned).
We are therefore interested in 
$\mathbb P(|A+\mathbf v\cdot Z_1^{(i)}|<\sqrt j/T)$. This is bounded above by 
$\mathbb P(|\mathbf v\cdot Z_1^{(i)}|<\sqrt j/T)$. More multivariate 
normal machinery tells us that the distribution
of $\mathbf v\cdot Z_1^{(i)}$ has the same distribution as $\|v\|$ times 
a standard normal random 
variable, so we want to estimate $\mathbb P(|Z_0|<\sqrt j/(T\|v\|))$,
where $Z_0$ is a standard normal random variable. Simple estimates 
show this is less than 
$\sqrt j/(T\|v\|)$, which, in turn, is bounded above by 
$\sqrt j/(\epsilon Ts_j(W^{(n)}))$. Hence,
$\bbp(\|\tilde Z^{-1}\|>T )\leq 
j^{3/2}/(\epsilon Ts_j(W^{(n)}))$.

Using \eqref{eq:Nunifcomp}, we obtain
\begin{equation}
\bbp\left (\|\tilde \Delta^{-1}\|>T \right)\leq 
\frac{(e\pi/2)^{d^2/2}j^{3/2}}{\epsilon Ts_j(W^{(n)})}
\end{equation}

\subsubsection{Final estimates}
We showed in \eqref{eq:sjWnbd} that 
$s_j(W^{(n)})\geq e^{nN(\lambda_j -2\tau)}$.
Recall also the expression for $B_n$ given in \eqref{eq:BnFinal}.
From \eqref{eq:Enbounds} and \eqref{def:MD}, we have the upper bounds: 
$\|M\|\leq e^{N n(\lambda_{j+1}+2\tau)}(1+\|A\|)$, after recalling that the columns of $E$ and $V$ are orthonormal.
Combining this with the results of \S\ref{sec:NormalBounds} 
and \eqref{eq:BnFinal} we obtain that 
\begin{equation}\label{eq:BoundBn}
 \begin{split}
    \bar\bbp(\|B_n -Y^{(n)}{W^{(n)}}^{-1}\|> T)&= 
    \bar\bbp(\|M\tilde\Delta^{-1}\|>T)\\
    &\leq \frac{j^{3/2}(e\pi/2)^{d^2/2}}
    {\epsilon e^{nN(\lambda_j-\lambda_{j+1}-5\tau)}T}
    + \bbp(\|A\|>e^{nN\tau}-1)\\
 \end{split}
\end{equation}

Recalling that $N=C|\log \ep|$ and $\|Y^{(n)}{W^{(n)}}^{-1}\|\leq 2/\del$,
and specialising 
\eqref{eq:BoundBn} to $T=1/\del$, we get
\begin{equation}
\begin{split}
  \bar\bbp(\|B_n \|> 3/\del)&\leq \bar\bbp(\|M\tilde \Delta^{-1}\|>1/\del)\\
  &\leq \frac{j^{3/2}\del}{e^{N(n(\lambda_j-\lambda_{j+1}-6\tau) -1/C)}}+
  \bar\bbp(\|A\|>e^{nN\tau}-1)<\chi/2,
  \end{split}\label{eq:BoundBnDelta}
\end{equation}
where we used (\ref{cond:eps}) and (\ref{cond:prob}) for the final inequality.

Fix $\oom\in\bar\Omega$ and suppose that $\omega\in\tilde G$ and 
$\|B_n \|\leq 3/\del$. Then Lemma \ref{lem:perpcalc}
shows that implies $\perp(\topspaceDSpert j{\ssig^{nN}\oom}, 
\bottomspacemat j{\cocycleps N{\ssig^{nN}\oom}}) 
\geq \del/6$.

We extract three conclusions from the fact that $\sig^{nN}\om\in G$.
Recall that $\del \leq \del_1=\chi/2$.
Lemma \ref{lem:LY}\eqref{it:contr} yields that
$\angle(\topspaceDSpert j{\ssig^{(n+1)N}\oom}, 
\topspacemat j{\cocycleps N{\ssig^{nN}\oom}}) \leq \chi/4$.
Next, the hypotheses of Lemma \ref{lem:LY} are satisfied with 
$A=\cocycle N{\sigma^{nN}\om}$ and $B=\cocycleps N{\sigma^{nN}\om}$.
Conclusion \eqref{it:spacecont} tells us that 
$\angle(\topspacemat j{\cocycleps N{\ssig^{nN}\oom}}, 
\topspacemat j{\cocycle N{\sig^{nN}\om}}) \leq \chi/4$.
Finally  we have $\angle(\topspaceDSunpert j{\sig^{(n+1)N}\om}, 
\topspacemat j{\cocycle N{\sig^{nN}\om}})<\chi/2$ from the definition of $G$.
Combining these we get
\[
 \angle\big(\topspaceDSpert j{\ssig^{(n+1)N}\oom}, 
 \topspaceDSunpert j{\sig^{(n+1)N}\om}\big)<\chi.
\]

By \eqref{eq:BoundBnDelta}, 
$\bar\bbp(\|B_n \|> 3/\del \ | \ \om \in \tld{G})<\chi/2$. Thus,
\[
 \bar\bbp(\angle\big(\topspaceDSpert j\oom,
 \topspaceDSunpert j\om\big)>\chi \ | \om)<\chi/2\text{ for all $\om\in\tilde G$.}
\]
We have therefore established \eqref{eq:toprove}, and Claim~\ref{claim:conv4good} is proved.

\end{proof}

\section*{Acknowledgments}
The research of GF and CGT is supported by an ARC Future Fellowship and an ARC Discovery Project (DP110100068). 
AQ acknowledges NSERC, ARC DP110100068 for travel support and UNSW for hospitality during a research
visit in 2012.

\bibliography{stochstab_refs}
\bibliographystyle{abbrv}
\end{document}